\documentstyle[amssymb,amsfonts]{amsart}

\def\be#1{\begin{equation}\label{#1}}
\def\bas{\begin{align*}}
\def\eas{\end{align*}}
\def\bi{\begin{itemize}}
\def\ei{\end{itemize}}

\theoremstyle{plain}
   \newtheorem{theorem}[subsection]{Theorem}

   \newtheorem{proposition}[subsection]{Proposition}
   
   \newtheorem{lemma}[subsection]{Lemma}
   \newtheorem{corollary}[subsection]{Corollary}

\newenvironment{proof}{\noindent {\bf Proof} }{\endprf\par}
\def \endprf{\hfill  {\vrule height6pt width6pt depth0pt}\medskip}
\def\emph#1{{\it #1}}
\def\textbf#1{{\bf #1}}

\theoremstyle{remark}

\theoremstyle{definition}

\begin{document}

\author[O. Bakas]{Odysseas Bakas}
\address{Odysseas Bakas: Department of Mathematics, Stockholm University, 106 91 Stockholm, Sweden}
\email{bakas@@math.su.se}

\author[E. Latorre]{Eric Latorre}
\address{Eric Latorre: ICMAT,
C/ Nicolás Cabrera, nº 13-15 Campus de Cantoblanco, UAM
28049 Madrid SPAIN. Partially supported by the grant  MTM2017-83496-P from the Ministerio de Ciencia e Innovaci{\'o}n (Spain) }
\email{eric.latorre@@icmat.es}

\author[D. Mart\'inez]{ Diana Cristina Rinc\'on Mart\'inez}
\address{Diana C. Rincon M.:
Institute of Mathematics, National Autonomous University of Mexico.
\'Area de la Investigaci\'on Científica, Circuito exterior, Ciudad Universitaria, 04510, Mexico, CDMX.}
\email{dcrinconm@@matem.unam.mx}

\author[J. Wright]{ James Wright }
\address{James Wright: School of Mathematics and Maxwell Institute for Mathematical Sciences, University of
Edinburgh, JCMB, King's Buildings, Peter Guthrie Tait Road, Edinburgh EH9 3FD, Scotland}
\email{J.R.Wright@@ed.ac.uk}

%\author{Odysseas Bakas, Eric Latorre, and James Wright}


%\thanks{The author was supported in part by an EPSRC grant.}
\title[Multiparameter oscillatory singular integrals]{A class of multiparameter oscillatory singular integral
operators: endpoint Hardy space bounds}

%\author{Odysseas Bakas, Eric Latorre, and James Wright}

\maketitle

\begin{abstract} We establish endpoint bounds on a Hardy space $H^1$ for
a natural class of multiparameter singular integral operators which {\it do not} decay away
from the support of rectangular atoms. Hence the usual argument via a Journ\'e-type
covering lemma to deduce bounds on product $H^1$ is not valid. 

We consider the class of multiparameter oscillatory singular integral operators given by
convolution with the classical multiple Hilbert transform kernel modulated by a general
polynomial oscillation. Various characterisations are known which give 
$L^2$ (or more generally $L^p, 1<p<\infty$) bounds. 
Here we initiate an investigation of endpoint bounds on the rectangular Hardy space $H^1$
in two dimensions; we give a characterisation when bounds hold
which are uniform over a given subspace of polynomials and somewhat surprisingly, we 
discover that the Hardy space and $L^p$ theories for these operators are very different.
\end{abstract}

%\maketitle

\section{Introduction}
There is a well developed connection between singular Radon transforms and 
oscillatory singular integral operators. For instance if $\Sigma$
is an $n$-dimensional surface given by the graph $\{(x, \Phi(x)) : x \in {\mathbb R}^n\}$
of a polynomial mapping $\Phi = (P_1, \ldots, P_k)$ where each 
$P_j \in {\mathbb R}[X_1,\ldots, X_n]$, then the so-called Hilbert transform along $\Sigma$,
$$
{\mathcal H}_{\Sigma, K} f(x, z) \ = \ p.v. \int_{{\mathbb R}^n} f(x - y, z - \Phi(y))
K(y) \, dy,
$$
has served as a model operator in the theory of singular Radon transforms. Here $K$ is a classical
Calder\'on-Zygmund kernel on ${\mathbb R}^n$. By computing the partial fourier transform in the
$z$ variable and using Plancherel's theorem, one sees that the $L^2$ boundedness of 
${\mathcal H}_{\Sigma, K}$
is equivalent to {\it uniform} $L^2$ boundedness of the oscillatory singular integral operator
$$
T_{\xi} \, g (x) \ = \ p.v. \int_{{\mathbb R}^n} g(x - y) \ 
e^{i \xi\cdot \Phi(y)} K(y) \, dy
$$
where we require uniformity in the frequency variable $\xi \in {\mathbb R}^k$. This
connection has been developed more deeply in \cite{RS1} and \cite{RS2}. It is well known that the operator 
${\mathcal H}_{\Sigma, K}$
is bounded on all $L^p$ with $1<p<\infty$ (see e.g. \cite{S}, Chapter XI ) but
a major open problem in the theory
of singular Radon transforms is to establish endpoint bounds; for example, to determine
whether or not ${\mathcal H}_{\Sigma, K}$ is weak-type $(1,1)$ or whether it is bounded 
on the appropriate real Hardy space $H^1$
(see \cite{C} for a result in this direction). This endeavour still seems to be a long term goal.

Although the boundedness properties of ${\mathcal H}_{\Sigma, K}$
and the uniform boundedness properties of $T_{\xi}$ are no longer equivalent
outside $L^2$, it has been of interest to
study operators of the form
$$
S_{P, K}\,  g (x) \ = \ p.v. \int_{{\mathbb R}^n} g(x - y) e^{i P(y)} K(y) \, dy
$$
for a general polynomial $P \in {\mathbb R}[X_1,\ldots, X_n]$ and determine whether they
are of weak-type $(1,1)$ or bounded on $H^1$.
A negative result would imply a negative outcome (and a positive result would give some
indication) for the corresponding singular Radon transforms. Since uniformity in the frequency
variable $\xi$ for the case $P(x) = \xi \cdot \Phi(x)$ where $\Phi$ is a general polynomial mapping
is required, one is naturally interested in $L^p$, weak-type $(1,1)$ and/or $H^1$ bounds
for $S_{P,K}$ which are {\it uniform}\footnote{There are other reasons for seeking such uniform
estimates; see e.g. \cite{S}} over the space of all polynomials of a fixed degree; that is,
bounds which are uniform in the coefficients of the polynomial oscillation $P$. 
This has been accomplished by a number of authors; for example, weak-type $(1,1)$ bounds by 
Chanillo and Christ \cite{CC} and Hardy space $H^1$ bounds
by Hu and Pan  \cite{Pan}, all bounds are uniform in the coefficients of the polynomial $P$.  

Recently the theory of singular Radon transforms has been extended to the multiparameter setting
and this was done for a number of reasons; see Street's monograph \cite{B} and the references
therein. This extension poses
a number of challenges in part because it is no longer the case that $L^2$ boundedness holds,
even when the underlying surface is polynomial. However we now have a good understanding
of the cancellation conditions needed to guarantee boundedness in various cases and 
furthermore, a general $L^p$
theory has been developed (see for example, \cite{NW}, \cite{D}, \cite{RS3}, \cite{B}, \cite{PY}, \cite{CWW1} and \cite{CWW2}). 
Needless to say, endpoint bounds for multiparameter singular Radon transforms are even more
challenging than the one parameter case which remains open. 

Exactly as in our discussion above, there is a connection between multiparameter singular
Radon transforms and multiparameter oscillatory singular integrals where now
the underlying Calder\'on-Zygmund $K$ has a multiparameter structure; for example, 
the multiple Hilbert transform kernel ${\mathcal K}(y) = 1/y_1\cdots y_n$. 
From the work of Ricci and Stein \cite{RS3} (via a simple lifting procedure), one can
determine precisely when $S_{P, {\mathcal K}}$ (equivalently $H_{\Sigma, {\mathcal K}}$)
is {\it uniformly} bounded on $L^2$. If 
$P(x) = \sum_{\alpha} c_{\alpha} x^{\alpha}$ is a real polynomial in $n$ variables,
we define the {\it support of $P$} as $\Delta_P = \{\alpha: c_{\alpha} \not= 0\}$.
For any finite subset $\Delta \subset {\mathbb N}_0^n$, let 
${\mathcal V}_{\Delta}$ denote 
%:= \{ P \in {\mathbb R}[X_1,\ldots, X_n] : \Delta_P \subseteq \Delta\}$
the finite dimensional subspace of real polynomials $P$ in $n$ variables with $\Delta_P \subseteq \Delta$.

{\bf Ricci-Stein Theorem} (\cite{RS3}) {\it Fix $\Delta \subset {\mathbb N}_0^n$. Then
\begin{equation}\label{RS-L2}
\sup_{P \in {\mathcal V}_{\Delta}} \|S_{P, {\mathcal K}} \|_{L^2 \to L^2} \ < \ \infty
\end{equation}
holds if and only if for every $\alpha = (\alpha_1, \ldots, \alpha_n) \in \Delta$, at least $n-1$
of the $\alpha_j$'s are even.}

There is an equivalent formulation for $H_{\Sigma, {\mathcal K}}$.
This result depends on our particular choice of multiparameter Calder\'on-Zygmund kernel 
${\mathcal K}(y) = 1/y_1\cdots y_n$. For a fixed polynomial $P \in {\mathbb R}[X,Y]$,
then a necessary and sufficient condition on $P$ is given in \cite{B2} so that $S_{P, K}$
is bounded on $L^2({\mathbb R}^2)$ for all multiparameter Calder\'on-Zygmund kernels $K$. 

When uniformity is not sought, there are a number of results which characterise those
individual polynomials $P$ for which $S_{P, {\mathcal K}}$ is bounded on $L^2$. Furthermore
these characterisations depend on how one truncates the operator $S_{P, {\mathcal K}}$.
For example, when $n=2$ such a characterisation was given in \cite{CWW1} for the
local operator (when the integration over $y \in {\mathbb R}^2$ is restricted to $|y|\le 1$)
and the characterisation is given in terms of the Newton diagram of $P$ which depends only
on the support $\Delta_P$. In \cite{P} a different characterisation (but still depending 
only on the support of $P$) was found for the global operator where the integration is taken over
all $y\in {\mathbb R}^2$. 

This is in sharp contrast to what happens in $n=3$ for
the corresponding triple Hilbert transform with a polynomial oscillation; in \cite{CWW2}, it was shown that
two polynomials $P$ and $Q$ may have the same support $\Delta_P = \Delta_Q$ yet 
$S_{P, {\mathcal K}}$
is bounded on $L^2$ whereas $S_{Q, {\mathcal K}}$ is {\it not} bounded on $L^2$! Here 
${\mathcal K}(y) = 1/y_1 y_2 y_3$, the triple Hilbert transform kernel. So when $n=3$, matters
are much more delicate but nonetheless a characterisation of $L^2$ boundedness 
was found in \cite{CWW2} and depends not only on the support of $P$ but also
on the parity of the coefficients. See also \cite{CK}
for other results in $n=3$.

Here we will be interested in examining how the multiparameter 
oscillatory singular integral operator $S_{P,{\mathcal K}}$
%defined with respect to the multiple
%Hilbert transform kernel $K(y) = 1/y_1 \cdots y_n$, 
acts on rectangular atoms. Recall
that a rectangular atom is an $L^2$ function $a_R$ supported in some rectangle $R$ (an $n$-fold
product of intervals) satisfying $\|a_R\|_{L^2} \le |R|^{-1/2}$ and possessing the cancellation property
$$
\int a_R (x_1,\ldots, x_{j-1}, y, x_{j+1}, \ldots, x_n) \, dy \ = \ 0
$$
for any $1\le j \le n$ and for almost every $x_1,\ldots, x_{j-1}, x_{j+1}, \ldots, x_n$. Given the
connection with multiparameter singular Radon transforms, we will be mainly interested
in {\it uniform} estimates and in particular we seek to understand when the estimate 
\begin{equation}\label{main-bound}
\int_{{\mathbb R}^n \setminus \gamma R} |S_{P, {\mathcal K}} \, a_R (x)| \, dx  \ \le \ C_{\gamma}
\end{equation}
holds
%not only uniformly for all rectangular atoms $a_R$ but also 
uniformly for all 
$P \in {\mathcal V}_{\Delta}$ for a fixed $\Delta \subset {\mathbb N}_0^n$.
Here $\gamma \ge 2$ and $\gamma R$ is the $\gamma$ dilate of $R$ with respect to its centre. 
If $S_{P, {\mathcal K}}$ is bounded on $L^2$, then an application of the Cauchy-Schwarz inequality
allows us to control $\|S_{P, {\mathcal K}} a_R\|_{L^1(\gamma R)}$. 

If there exists an $\epsilon>0$ such that $C_{\gamma} \le C_{\epsilon} \gamma^{-\epsilon}$
holds for some $C_{\epsilon}$ and all $\gamma\ge 2$, then\footnote{here $H^1_{\rm prod}$
is the natural real Hardy space associated to multiparameter dilations/structure.}
$S_{P, {\mathcal K}} : H^1_{\rm prod}({\mathbb R}^n) \to L^1({\mathbb R}^n)$, assuming
that $S_{P, {\mathcal K}}$ is also bounded on $L^2$. Furthermore, the $H^1_{\rm prod} \to L^1$
operator norm of $S_{P,{\mathcal K}}$ depends only on $C_{\epsilon}$ and its $L^2$ operator norm.
This result depends on a Journ\'e-type covering lemma for rectangles
and is due to  R. Fefferman \cite{RF} in the two parameter setting and J. Pipher \cite{JP} in the general multiparameter setting.

Interestingly any $\gamma$ decay bound in \eqref{main-bound} is {\it false} for
oscillatory singular integral operators, even in the one parameter setting, $n=1$ (see 
Section \ref{decay-failure} below). 
This explains our interest in obtaining bounds first on the {\rm rectangular} Hardy space
$H^1_{\rm rect}({\mathbb R}^n)$, the atomic space constructed 
from rectangular atoms.\footnote{elements in
$H^1_{\rm prod}$ also have an atomic decomposition but the atoms are more complicated, associated
to arbitrary open sets of finite measure.}
Hence bounds on $H^1_{\rm prod}({\mathbb R}^n)$, if true, requires a new, alternate approach and we leave this
for a future investigation.  

Our goal is to characterise those finite sets $\Delta \subset {\mathbb N}_0^n$ such that
\eqref{main-bound} holds uniformly for all $P\in {\mathcal V}_{\Delta}$ with a constant $C_{\gamma}
= C_{\gamma, \Delta}$
only depending on $\gamma$ and $\Delta$ (and of course, independent on the rectangular atom
$a_R$). By accomplishing this, we can then import any of the many $L^2$ results known for
$S_{P, {\mathcal K}}$, uniform or otherwise, and obtain boundedness 
from $H^1_{\rm rect}$  to $L^1$. But we highlight the Ricci-Stein Theorem which gives
us a characterisation of when {\it uniform} $L^2$ bounds hold and so, together with \eqref{main-bound},
would give us uniform bounds on $H^1_{\rm rect}$. 

In this paper we provide such a characterisation in two dimensions, when $n=2$.
First of all, without loss of generality,
we may assume $(0,0) \notin \Delta$. Furthermore when
$\Delta \subset {\mathbb N}_0^2$,
we set $\Delta_j = \{ k\ge 0 : (j,k) \in \Delta \}$ and $\Delta^k = \{j\ge 0 : (j,k) \in \Delta \}$.
\begin{theorem}\label{2-d} If $\Delta \subset {\mathbb N}_0^2$ with $(0,0)\notin \Delta$, then
\eqref{main-bound} holds uniformly for all $P\in {\mathcal V}_{\Delta}$ if and only if
\begin{equation}\label{nec-suff}
(a) \ (1,0) \ {\rm and} \ (0,1) \notin \Delta, \ \ \ {\rm and} \ \ \ (b) \ \ |\Delta_0| |\Delta_1| + |\Delta^0|
|\Delta^1| \ = \ 0.
\end{equation}
\end{theorem}
Condition (a) is well known to be a necessary condition for any boundedness result on $H^1$ for 
oscillatory singular integral operators, even in the one parameter setting. Condition (b)
is the new, interesting necessary condition for this 2 parameter case.  Assuming condition
(a) holds, we see that condition (b)
fails precisely when there exist a $(0,k_0) \in \Delta$ with $k_0\ge 2$ AND there is a $(1, k_1) \in \Delta$ 
for some $k_1 \ge 1$ (or the corresponding situation holds with the coordinates swapped). 
In particular if $P(s,t) = c s t^2 + d t^4$, then the Ricci-Stein Theorem shows that $S_{P, {\mathcal K}}$
is bounded on $L^2$ (and in fact on all $L^p, 1<p<\infty$) with bounds which are uniform in $c$ and $d$. However by Theorem
\ref{2-d} this is {\it not} the case on $H^1_{\rm rect}$, showing a difference in the $L^p$
and Hardy space theories for this class of singular integral operators. 

We can combine Theorem \ref{2-d} with the Ricci-Stein Theorem to obtain a characterisation for uniform boundedness from
$H^1_{\rm rect}({\mathbb R}^2)$ to $L^1({\mathbb R}^2)$. First, we observe that if
 $S_{P,{\mathcal K}}:
H^1_{\rm rect}({\mathbb R}^n) \to L^1({\mathbb R}^n)$ is bounded uniformly for 
$P\in {\mathcal V}_{\Delta}$, then necessarily $S_{P, {\mathcal K}}$ is bounded on 
$L^2({\mathbb R}^n)$, uniformly for $P\in {\mathcal V}_{\Delta}$ (this follows from a
standard argument, see for example \cite{L}) and so $\Delta \subset {\mathbb N}_0^n$
necessarily satisfies the condition that every $\alpha \in \Delta$ has at least $n-1$ 
even components. 
\begin{corollary}\label{2-d-bound} Let $\Delta \subset {\mathbb N}_0^2$ and assume, without loss
of generality, $(0,0)\notin \Delta$. Then $S_{P, {\mathcal K}}: H^1_{\rm rect}({\mathbb R}^2)
\to L^1({\mathbb R}^2)$ is bounded uniformly for $P\in {\mathcal V}_{\Delta}$ if and only if
$j k$ is even for every $(j,k) \in \Delta$ AND condition
\eqref{nec-suff} holds.
\end{corollary} 

{\bf Notation} Uniform bounds for oscillatory integrals lie at the heart of this paper. Keeping track of constants
and how they depend on the various parameters will be important for us. For the most part, constants $C$
appearing in inequalities $P \le C Q$ between positive quantities $P$ and $Q$ will be {\it absolute} or
{\it uniform} in that they can be taken to be independent of the parameters of the underlying problem. 
We will use $P \lesssim Q$ to denote $P \le C Q$ and $P \sim Q$ to denote $C^{-1} Q \le P \le C Q$.
If $P$ is a general real or complex quanitity,
we write $P = O(Q)$  to denote $|P| \le C Q$ and when we want to highlight a dependency on a parameter
$\gamma$, we write $P = O_{\gamma}(Q)$ to denote $|P| \le C_{\gamma} Q$.
%When
%we allow the constant $C$ to depend on a parameter (or parameters), say $\omega \in {\mathbb S}^{d-1}%$,
%we will write $A \le C_{\omega}B$ or $A \lesssim_{\omega} B$. 

We will use multi-index notation: if $\alpha = (j,k) \in {\mathbb N}_0^2$
and $x = (x_1, x_2) \in {\mathbb R}^2$, we denote $x^{\alpha}$ as the
monomial $x_1^{j} x_2^{k}$ and we use the notation
$$
\partial^{\alpha} \phi (x) \ = \ 
\frac{\partial^{j+k} \phi}{\partial x_1^{j} \partial x_2^{k}} (x) 
$$
to denote the associated partial derivative. We also write $|\alpha| = j+k$.

\section{Failure of decay in \eqref{main-bound}}\label{decay-failure} Here we prove that there is no decay
in $\gamma$ in the bound \eqref{main-bound} for the class of oscillatory singular integral
operators, even in the one parameter case. Hence one cannot establish bounds
on $H^1_{\rm prod}$ for this class of singular integral operators 
by the usual method via a Journ\'e-type covering lemma.

We begin with the most classical oscillatory singular
integral operator
$$
Tf(x) \ = \ \int_{{\mathbb R}} e^{i(x-y)^2} \frac{1}{x-y} f(y) \, dy
$$
and prove the following.
\begin{proposition}\label{failure} There does not exist an $\epsilon>0$ such that 
$$
\int_{\gamma |I| \le |x|} |T a_I (x) | \, dx \ \le \ C \gamma^{-\epsilon}
$$
holds for some $C$, every $\gamma\ge 2$ and all atoms $a_I$ associated with
intervals $I$.
\end{proposition}
\begin{proof} We simply consider intervals $I = [-1/2 |I|, 1/2 |I|]$ for small $|I| \ll 1$
and take
$a_I (s) = e^{-i s^2} b_I(s)$ where
$b_I(s) = 1$ when $0\le s \le |I|/2$ and $b_I(s) = -1$ for $-1/2 |I| \le s < 0$. On easily
checks that $a_I$ is an atom associated with the interval $I$. 
We will
take $\gamma = |I|^{-2}$ and show that 
\begin{equation}\label{no-decay-bound}
\int_{\gamma |I| \le |x|} |T a_I (x)| \, dx \ \gtrsim \ 1
\end{equation}
which will establish the proposition. For this atom $a_I$, we add and subtract $1/x$ in
the definition of $T a_I (x)$ to conclude that
$$
\int_{\gamma |I| \le |x|} |T a_I (x) | \, dx \ \ge \ \int_{\gamma |I| \le |x|} \frac{1}{|x|} \
\Bigl| \int_{\mathbb R} e^{i (x-s)^2} a_I (s) ds \Bigr| \ dx \ - \ 2 \gamma^{-1}
$$
where we take $\gamma = |I|^{-2} \gg 1$. However 
$e^{i(x-s)^2} a_I (s) = e^{i x^2} e^{- 2i x s} b_I(s)$ and so
$$
\Bigl| \int_{\mathbb R} e^{i (x-s)^2} a_I (s) ds \Bigr| \ = \ \bigl| {\widehat b_I}(2 x) \bigr| \
= \ \frac{ |\cos( x |I|) - 1 |}{|x| |I|} \ \gtrsim \ \frac{1}{|x| |I|}
$$
holds for any $x$ satisfying $| |x| |I| - k \pi/2 | \le \pi/200$ for some odd $k\ge 1$. Therefore
when $\gamma = |I|^{-2}$, we have
$$
\int_{\gamma |I| \le |x|} \frac{|{\widehat b_I}(x)|}{|x|}  \, dx \ \gtrsim \ \sum_{k: k \, {\rm odd}} 
\int_{E_k} \frac{|{\widehat b_I}(x)|}{|x|} \, dx \ \gtrsim \ 
\frac{1}{|I|} \, \sum_{k: k \, {\rm odd}} \int_{E_k} \frac{1}{x^2} \, dx
$$
where $E_k = \{ x: ||x| |I| - k \pi/2| \le \pi/200 \}$. Since $|E_k| \sim |I|^{-1}$ and
$|x| \sim k/|I|$ 
for $x \in E_k$, we have
$$
\frac{1}{|I|} \, \sum_{k: k \, {\rm odd}} \int_{E_k} \frac{1}{x^2} \, dx \ \sim \ 
|I| \, \sum_{k: \, {\rm odd}} \frac{1}{k^2} \, |E_k| \ \gtrsim \ \sum_{k: \, {\rm odd}} \frac{1}{k^2}
\ \gtrsim \ 1,
$$
establishing \eqref{no-decay-bound} as desired.
\end{proof}

From Proposition \ref{failure} we can easily construct examples in higher dimensions
simply by taking $n$-fold products.

\section{A more robust formulation and some preliminaries}\label{prelims}
We fix a finite set $\Delta \subset {\mathbb N}_0^2$ satisfying condition \eqref{nec-suff} in
Theorem \ref{2-d}. We
also fix a $P(x) = \sum c_{\alpha} x^{\alpha}$ with $\Delta_P \subseteq \Delta$
but we keep in mind that our estimates should always be independent of $P \in {\mathcal V}_{\Delta}$.

Let $\phi \in C^{\infty}_0({\mathbb R})$ be an even function which is supported in $\{|s|\sim 1\}$
and has the property that
$\sum_{p\in {\mathbb Z}} \phi(2^{-p} s) \equiv 1$ for all $s\in {\mathbb R}\setminus\{0\}$. 
Set $\psi_p (s) = \phi(2^{-p} s)/s$ and for ${\bf p} = (p, q) \in {\mathbb Z}^2, \
y = (y_1, y_2) \in {\mathbb R}^2$, 
we write $\psi_{\bf p}(y) = \psi_{p}(y_1)\psi_{q}(y_2)$ and
$$
T_{\bf p} f(x)  := \ \int_{{\mathbb R}^2} \psi_{\bf p}(y) e^{i P(y)} f(x-y) \, dy.
$$
For any finite subset ${\mathcal F} \subset {\mathbb Z}^n$, we consider
the following general truncation of our operator $S_{P, {\mathcal K}}$,
$$
T_{\mathcal F} f(x) \ := \ \sum_{{\bf p}\in {\mathcal F}} T_{{\bf p}} f(x).
$$
Our main goal is to prove the bound \eqref{main-bound} for $T_{\mathcal F}$,
uniformly for all finite subsets ${\mathcal F}$. This implies a more robust version
of Theorem \ref{2-d}. We note that the Ricci-Stein Theorem
also holds uniformly for all such truncations. By translation invariance and since we seek
bounds which hold uniformly for all $P\in {\mathcal V}_{\Delta}$, we may assume,
without loss of generality, that the support of the rectangular atom $a_R$
is the unit square; that is, matters are reduced to showing that for $\gamma\ge 2$,
\begin{equation}\label{reduced-main-bound}
\int_{|x| \ge \gamma} |T_{\mathcal F} a(x) | \, dx \ \le \ C_{\gamma}
\end{equation}
holds uniformly for all atoms $a$ supported in the unit square, for all $P \in {\mathcal V}_{\Delta}$
and for all finite subsets ${\mathcal F}\subset {\mathbb Z}^2$.  

We now give a few useful results which we will use time and time again.

For $Q(x) = \sum d_{\alpha} x^{\alpha} \in {\mathcal V}_{\Delta}$ define
$\|Q\|_1 = \sum |d_{\alpha}|$ and for some fixed $C_0\ge 1$, set
$$
||| Q ||| \ := \ \max_{\alpha\in \Delta} \inf_{x\in [-C_0, C_0]^n} |\partial^{\alpha} Q(x)|.
$$

\begin{lemma}\label{equiv-norms} Let $\Delta \subset {\mathbb N}_0^n$ be a finite subset
with ${\vec 0} \notin \Delta$. For $C_0 \ge 1$, define $|||\cdot |||$ as above. Then 
there is a positive constant $C>0$, dependingly only on $\Delta$, $C_0$ and $n$ such that
\begin{equation}\label{equiv-norms-inequality}
|||Q||| \ \ge \ C \, \|Q\|_1
\end{equation}
holds for every  $Q \in {\mathcal V}_{\Delta}$.
\end{lemma}

\begin{proof} The proof is just the usual {\it equivalence of norms} argument although
$|||\cdot|||$ is not a norm (the triangle inequality fails). However it does act enough like a norm
to make the usual argument work.

Note that $|||\lambda Q||| = |\lambda| |||Q|||$ for any scalar $\lambda \in {\mathbb R}$
and from this we see that \eqref{equiv-norms-inequality} holds with
$$
C \ := \ \inf_{Q \in S_1} |||Q |||
$$
where $S_1 = \{ Q \in {\mathcal V}_{\Delta} : \|Q\|_1 = 1\}$. It suffices to show that
$C$ is positive. Suppose $C = 0$. Since $S_1$ is the unit sphere in
the finite dimensional vector space ${\mathcal V}_{\Delta}$ with respect to the norm $\|\cdot\|_1$,
it is compact and so we can find a sequence $Q_j \in S_1$ such that $\|Q_j - Q\|_1 \to 0$
for some $Q\in S_1$ and such that $|||Q_j||| \to 0$. We will see that this implies $Q =0$ which
gives us our contradiction since $Q \in S_1$ and hence nonzero. 

First we observe that for every $\alpha \in \Delta$, the corresponding
coefficient $d_{\alpha}^j$ of $Q_j$ tends to zero. This follows from $|||Q_j||| \to 0$
by a simple induction argument, starting with those $\alpha_0 \in \Delta$
satisfying $|\alpha_0| = \max_{\alpha \in \Delta} |\alpha|$ and hence 
$\partial^{\alpha_0} Q_j (x) \equiv d_{\alpha_0}^j \alpha_0 !$.
But since $\|Q_j - Q\|_1 \to 0$,
we see that $d_{\alpha}^j$ converges to  $d_{\alpha}$,
the corresponding coefficient of $Q$. Hence $d_{\alpha} = 0$ for every $\alpha \in \Delta$
and so $Q = 0$.
\end{proof}

We will use Lemma \ref{equiv-norms} to estimate oscillatory integrals with polynomial
phases. In fact we will use Lemma \ref{equiv-norms} in combination with the following
higher dimensional version of van der Corput's lemma.
 
\begin{lemma}\label{UL-1} Let $\Delta$ be a finite subset of ${\mathbb N}_0^n$
such that ${\bf{0}} \notin \Delta$. Then for every $C_0 > 0$ and $\psi \in C^{\infty}_0({\mathbb R}^n)$
with ${\rm supp}(\psi) \subset [-C_0, C_0]^n$, there is a 
$\delta$ with $0<\delta<1$ and $C$, both depending only on $|\Delta|, C_0$ and $n$, such that
whenever we have a uniform bound from below $|\partial^{\alpha} Q(x)| \ge \lambda$ on
the support of $\psi$
for some derivative $\alpha \in \Delta$ of an element $Q\in {\mathcal V}_{\Delta}$, then
\begin{equation}\label{VC}
\Bigl| \int_{{\mathbb R}^n} e^{i Q(x)} \, \psi(x) \,  dx \Bigr| \ \le \ C \lambda^{-\delta} 
(\|\psi\|_{L^{\infty}} + \|\nabla \psi\|_{L^1} )
\end{equation}
holds.
\end{lemma}

For our applications, the importance of this lemma lies in the uniformity in the bound \eqref{VC},
the fact that the constant $C$ depends only on $\Delta, C_0$ and $n$ and otherwise
can be taken to be independent of $Q$ and $\lambda$. Due to this uniformity,  
the proof does not quite follow from the standard higher dimensional version of the
classical van der Corput's lemma as found for instance in \cite{S}, Proposition 5 page 342,
since we do not necessarily have uniform control of the $C^k$ norms of $Q(x)/\lambda$. This would
be the case IF $\lambda$ is comparable $\|Q\|_1$, and although by Lemma \ref{equiv-norms}
we can always find a $\beta \in \Delta$ so that the uniform bound
$|\partial^{\beta}Q(x)| \gtrsim \|Q\|_1$ holds on the support of $\psi$ (and hence
the result in \cite{S} would imply the bound in \eqref{VC} with $\lambda = \|Q\|_1$), our applications
combining Lemmas \ref{equiv-norms} and \ref{UL-1} are somewhat nonstandard. 

At times our arguments will 
have the following format: given a polynomial phase $\Phi \in {\mathcal V}_{\Delta}$ whose
corresponding oscillatory integral given in \eqref{VC} is the object we would like to bound, it will not
be clear how to successfully estimate $\|\Phi\|_1$ from below. Nevertheless, we will be able pass 
to a related polynomial $Q= Q_\Phi$ whose norm $\|Q\|_1$ can be effectively bounded below
and furthermore, we will be able to relate derivatives of $Q$ to derivatives of $\Phi$. We will
apply Lemma \ref{equiv-norms} to $Q$ to find a derivative of $Q$ bounded below by
$\|Q\|_1$ and then deduce a derivative bound for $\Phi$ in terms of $\|Q\|_1$.
We will then apply Lemma \ref{UL-1} to $\Phi$ with $\lambda = \|Q\|_1$. The two norms
$\|\Phi\|_1$ and $\|Q\|_1$ will {\bf not} be comparable in general.  

{\bf Proof of Lemma \ref{UL-1} }
The bound \eqref{VC} follows from a higher dimensional
version of  van der Corput's lemma found in \cite{CCW}, Proposition 4.14 on page 1004,
whose hypotheses are satisfied for polynomials with bounded degree with a concluding
bound which has the desired uniformity. 

In fact the bound given in Proposition 4.14 in \cite{CCW} is
$$
\Bigl| \int_{{\mathbb R}^n} e^{i Q(x)} \, \psi(x) \,  dx \Bigr| \ \le \ C \lambda^{-1/|\alpha|} 
(\|\psi\|_{L^{\infty}} + \|\nabla \psi\|_{L^1} )
$$
but since the bound $\|\psi\|_{L^1}$ trivially holds, we see that \eqref{VC} holds with
$\delta = \min( \delta^{*}, 1/2)$ where $\delta^{*} = \max( 1/|\alpha| : \alpha \in \Delta)$.
This completes the proof of Lemma \ref{UL-1}.

As an application of Lemmas \ref{equiv-norms} and \ref{UL-1}, with the format described above,
we derive an $L^2$ bound for $T_{\bf p}$.
More precisely, since we are interested only in how $T_{\bf p}$ acts on atoms supported
in the unit square, we consider the operator
$$
{\tilde T}_{\bf p} f(x) \ := \ \int_{{\mathbb R}^2} \psi_{\bf p}(x - y) e^{i P(x-y)} \varphi(y) f(y) \, dy
$$
for some $\varphi \in C^{\infty}_0({\mathbb R}^2)$ supported in $[-3,3]^2$
with $\varphi(x) \equiv 1$ for all $x$ in the
unit square. We will apply the above two lemmas to deduce a bound for the kernel
of ${\tilde T}^{*}_{\bf p} {\tilde T}_{\bf p}$ which in turn will give us a bound on the $L^2$ operator
norm of ${\tilde T}_{\bf p}$. 

For ${\bf p} = (p, q) \in {\mathbb Z}^2$ and $\alpha = (j,k) \in {\mathbb N}_0^2$,
we use the notation ${{\bf p}\cdot \alpha} = p j + q k$ and
$2^{\bf p} \circ x^{\alpha} = (2^{p} x_1)^{j} (2^{q} x_2)^{k}$.

\begin{proposition}\label{T*T} Let ${\bf p} = (p, q) \in {\mathbb Z}^2$ be ordered, 
$p\le q$. Then for some $0 < \delta < 1$,
\begin{equation}\label{L2-bound}
\|{\tilde T}_{\bf p}\|_{L^2} \ \lesssim \ 
\begin{cases}
2^{- q/2} \ \ \ |c_{\alpha_{*}} \, 2^{p j_{*} + q k_{*}} |^{-\delta} & \mbox{\rm if} \ p\le 0 \\
2^{-(p+q)/2} \ |c_{\alpha_{*}} \, 2^{p(j_{*} -1) + q k_{*}} |^{-\delta} & \mbox{\rm if} \ p\ge 0 
\end{cases}
\end{equation}
where $P(x) = \sum c_{\alpha} x^{\alpha}$ and  $\alpha_{*} = (j_{*}, k_{*})$ is any element 
in $\Delta_P$ with $j_{*}\ge 1$.  
\end{proposition}

\begin{proof} The kernel $L$ of ${\tilde T}^{*}_{\bf p} {\tilde T}_{\bf p}$ is
$$
L(x,u) \ = \ \varphi(x)\varphi(u) \int_{{\mathbb R}^2} e^{i [P(y - x) - P(y-u)]} \psi_{\bf p}(y-x) \psi_{\bf p}(y-u) \, dy.
$$
Note that $L$ is supported in $[-3, 3]^4$ and if $p\le 0$, $L$ is further supported when
$|x_1 - u_1| \le 2^{p}$.
We make the change of variables $y \to 2^{\bf p}\circ y$ to conclude
$$
L(x,u) \ = \ \varphi(x)\varphi(u) 2^{-|{\bf p}|} \int_{{\mathbb R}^2} e^{i \Phi(y)} \, 
\Theta(y) \, dy.
$$
where 
$$
\Phi(y) = \Phi_{{\bf p}, x, u}(y) = \sum c_{\alpha} 2^{{\bf p}\cdot \alpha} 
\bigl[ (y - 2^{-{\bf p}}\circ x)^{\alpha} - (y - 2^{-{\bf p}}\circ u)^{\alpha} \bigr]
$$
and
$$
\Theta(y) = \Theta_{{\bf p}, x,u} (y) =  \frac{\phi(y_1 - 2^{-p} x_1)}{y_1 - 2^{-p} x_1}
\frac{\phi(y_1 - 2^{-p} u_1)}{y_1 - 2^{-p} u_1}
 \frac{\phi(y_2 - 2^{-q} x_2)}{y_2 - 2^{-q} x_2}
\frac{\phi(y_2 - 2^{-q} u_2)}{y_2 - 2^{-q} u_2}
$$
is a smooth function, supported in $[-5,5]^2$ with uniformly bounded $C^k$ norms. Let
$$
g(t) \ := \ \sum c_{\alpha} 2^{{\bf p}\cdot \alpha} (y - 2^{-{\bf p}}\circ x + t 2^{-{\bf p}}
\circ (x - u))^{\alpha} 
$$
and note that $\Phi(y) = g(1) - g(0) = \int_0^1 g'(t) dt$. Writing 
$$
X \ := \  y_1 - 2^{-{p}}x_1 + t 2^{-{p}} (x_1 - u_1) \ = \ 
(1-t) (y_1 - 2^{-p} x_1) + t (y_1 - 2^{-p} u_1),
$$
and similarly for $Y$, 
we see that $X, Y \in [-10,10]^2$ for all $t\in [0,1]$. Also $\Phi(y) =$
$$
\int_0^1 \Bigl[\sum_{\alpha \in \Delta_P} c_{\alpha} 2^{{\bf p}\cdot \alpha} 
j 2^{-p} (x_1 - u_1) X^{j -1} Y^k + \sum_{\alpha \in \Delta_P} c_{\alpha} 2^{{\bf p}\cdot \alpha} 
k 2^{-q} (x_2 - u_2) X^{j} Y^{k-1}
 \Bigr] \, dt
$$
$$
= \ \int_0^1 \sum_{(j,k) \in {\tilde \Delta}_P} 2^{pj + q k} \Bigl[ 
c_{j+1,k} 
(j +1) (x_1 - u_1) + c_{j,k+1} (k +1) (x_2 - u_2)
 \bigr] X^j Y^k \Bigr] \, dt
$$
where ${\tilde \Delta}_P =  (\Delta_P - (1,0))\cup(\Delta_P - (0,1))$. We note that
${\tilde \Delta}_P \subseteq {\tilde \Delta}$ where ${\tilde \Delta} = (\Delta - (1,0))\cup (\Delta - (0,1))$
and every $(j,k) \in {\tilde \Delta}$  satisfies $j+k \ge 1$. We now apply Lemma \ref{equiv-norms}
to $Q(X,Y) = \sum_{(j,k) \in {\tilde \Delta}_P} d_{j,k} X^j Y^k$ where
$$
d_{j,k} \ = \  2^{p j + q k} [c_{j+1,k}  (j +1) (x_1 - u_1) + c_{j,k+1} (k+1) (x_2 - u_2)]
$$
and ${\tilde \Delta}$ to find a derivative $\alpha = (j,k)\in {\tilde \Delta}$ such that 
$|\partial^{\alpha} Q(X,Y)| \gtrsim \|Q\|_1$ for $(X,Y) \in [-10,10]^2$. 

Hence $\partial^{\alpha} Q(X,Y)$ is single-signed on $[-10, 10]^2$ and so
$$
\bigl| \partial^{\alpha}_y \Phi(y) \bigr| = \int_0^1 |\partial^{\alpha}_{X,Y} Q(X,Y)| \, dt \ \gtrsim
\ \|Q\|_1
$$
holds for all $y = (y_1, y_2)$ in the support of $\Theta$. Here we used the fact that $X$ and $Y$ are  translates of $y_1$ and $y_2$; $X = y_1 + B_1, Y = y_2 + B_2$ for some $B_1, B_2$.  

Using the fact that $j_{*}\ge 1$, we see that $\|Q\|_1 \ge |d_{j_{*} - 1, k_{*}}| = $
$$
2^{ p (j_{*} - 1) + q k_{*}} \,
\bigl| c_{j_{*}, k_{*}}  j_{*} 
(x_1 - u_1) +  c_{j_{*} - 1, k_{*} + 1} 
(k_{*} + 1) (x_2 - u_2) \bigr|
$$
$$
\ge  \ \ | c_{j_{*}, k_{*}}| 2^{p (j_{*} - 1) + q k_{*}} \bigl| x_1 - u_1 + B(x_2, u_2) \bigr|
$$
where $B(x_2,u_2)$ depends only on $x_2, u_2$ and the coefficients of $P$. 

We now apply Lemma \ref{UL-1} to $\Phi$ and $\lambda = \|Q\|_1$ to deduce the existence
of a $\delta = \delta(\Delta)$ with $0<\delta < 1$ such that 
\begin{equation}\label{L}
|L(x,u)| \ \lesssim \ 2^{-(p+q)} \,
 |c_{j_{*}, k_{*}} 2^{p (j_{*} - 1) + q k_{*}} (x_1 - u_1 + B(x_2, u_2) )|^{-\delta} .
\end{equation}
Since $\int_{|x |\le 3} |x_1 - u_1 + B(x_2, u_2)|^{-\delta} dx \le C_{\delta}$, we have
$$
\sup_u  \int |L(x,u)| \, dx \ \lesssim \ 2^{-(p+q)} 
\, |c_{j_{*}, k_{*}} 2^{p (j_{*} - 1) + q k_{*}}|^{-\delta}.
$$
Similarly $\sup_x  \int | L(x,u) | du \lesssim  
2^{-(p+q)} |c_{j_{*}, k_{*}} 2^{p (j_{*} - 1) + q k_{*}}|^{-\delta}$
and hence an application of Schur's lemma shows
$$
\|{\tilde T}^{*}_{\bf p} {\tilde T}_{\bf p} \|_{L^2} \ \lesssim \ 
 2^{-(p+q)} |c_{j_{*}, k_{*}} 2^{p (j_{*} - 1) + q k_{*}}|^{-\delta},
 $$
 implying
 $$  
 \|{\tilde T}_{\bf p}\|_{L^2} \ \lesssim \ \ 2^{-(p+q)/2} 
|c_{j_{*}, k_{*}} 2^{p (j_{*} - 1) + q k_{*}}|^{-\delta/2}
 $$ 
 which proves the Proposition for the $p\ge 0$ case.
 
 When $p\le 0$, we use the fact that $L(x,u)$ is supported in $E$ where
 $E = \{(x,u) \in [-3,3]^4: |x_1 - u_1| \le 2^p \}$. Using the bound
 \eqref{L} for $L(x,u)$, integrating over $E$ and making the change of variables
 $x_1 \to 2^{-p} (x_1 - u_1)$ we have 
 $$
 \int_{E} |L(x,u)| \, dx \ \lesssim \ 2^{-(p+q)} 2^{p}
 |c_{j_{*}, k_{*}} 2^{p (j_{*} - 1) + q k_{*}}|^{-\delta}
  \int_{|x|\le 3}  |2^{p} x_1  +  B(x_2, u_2) |^{-\delta} dx .
  $$
 Since
$$
   \int_{|x|\le 3}  |2^{p} x_1  +  B(x_2, u_2) |^{-\delta} dx = 
 2^{-\delta p} \int_{|x|\le 3}  |x_1  +  2^{-p} B(x_2, u_2) |^{-\delta} dx  \lesssim 
 2^{-\delta p},
$$
we have
$$
\int |L(x,u)| \, dx \lesssim 2^{-q} 2^{-\delta p}
 |c_{j_{*}, k_{*}} 2^{p (j_{*} - 1) + q k_{*}}|^{-\delta} \ = \
 2^{-q} |c_{j_{*}, k_{*}} 2^{p j_{*} + q k_{*}}|^{-\delta}.
 $$
As above, this leads to the bound
$\|{\tilde T}^{*}_{\bf p} {\tilde T}_{\bf p} \|_{L^2} \lesssim  
 2^{-q} |c_{j_{*}, k_{*}} 2^{p j_{*} + q k_{*}}|^{-\delta}$ and  hence
 $$  
 \|{\tilde T}_{\bf p}\|_{L^2} \ \lesssim \ \ 2^{-q/2} \,
|c_{j_{*}, k_{*}} 2^{p j_{*} + q k_{*}}|^{-\delta/2}
 $$ 
 which finishes the proof of the Proposition.
\end{proof}

We end this section with a final useful lemma.

\begin{lemma}\label{UL-2} Let ${\mathcal P}_d$ be the collection of real polynomials of a single variable
of degree at most $d$,
and let ${\mathcal G} \subset {\mathbb Z}$ be a finite set of integers. Then
$$
C_d \ := \ 
\sup_{Q\in {\mathcal P}_d, {\mathcal G}} \
\Bigl|\sum_{p\in {\mathcal G}} \ \int_{\mathbb R} \psi_p(s) \, e^{i Q(s)} \, ds \ \Bigr| 
$$
is finite.
\end{lemma}

This is a well known result; see for example \cite{S}, Chapter XI, page 513.

\section{Proof of Theorem \ref{2-d} -- Prelude to the sufficiency part}\label{prelude}
As stated in the previous section, we will
establish a more robust version of Theorem \ref{2-d} by showing that the
uniform bound \eqref{main-bound} holds for $T_{\mathcal F}$ where ${\mathcal F}$ is any finite subset 
of ${\mathbb N}_0^2$. Without loss of generality we may take the elements
${\bf p} = (p,q) \in {\mathcal F}$ to be ordered, say $p\le q$. 
Furthermore, since we are proving
$L^1$ bounds away from the
unit square, for $|x|\ge \gamma$, it suffices to consider a finite ${\mathcal F}$ 
with every ${\bf p} = (p, q)  \in {\mathcal F}$ satisfying $p\le q$
{\it and} $q \ge c_{\gamma} \gg 1$. For such ${\mathcal F}$, we see
that for any atom $a$ supported in the unit square,
$T_{\bf p}a(x)$ with ${\bf p}\in {\mathcal F}$ is automatically supported
where $|x|\ge \gamma$ and so it suffices to prove
\begin{equation}\label{reduction}
\int |T_{\mathcal F} a (x)| \, dx \ \lesssim \ 1,
\end{equation}
uniformly for all atoms $a$ supported in the unit square and all such ${\mathcal F}$
described above.

We decompose such an ${\mathcal F}$ into $O_{|\Delta|}(1)$ disjoint sets
such that
\begin{equation}\label{king}
|c_{\alpha_0}| \ 2^{{\bf p}\cdot \alpha_0} \ \ge \ |c_{\alpha}| \ 2^{{\bf p}\cdot \alpha}
\end{equation}
holds for some $\alpha_0 \in \Delta$ and all $\alpha \in \Delta$. It suffices to consider a fixed 
subset ${\mathcal F}_0$ where \eqref{king} holds, say for $\alpha_0 = (j_0, k_0) \in \Delta$, and establish
\eqref{reduction} with ${\mathcal F}$ replaced by ${\mathcal F}_0$. 

\subsection*{The case $j_0\ge 1$}
First we will consider the case $j_0 \ge 1$. In this case, for
${\bf p} \in {\mathcal F}_0$, we consider the difference operator
$D_{\bf p} = T_{\bf p} - S_{\bf p}$ where
$$
S_{\bf p} f(x) \ := \ \int_{{\mathbb R}^2} \psi_{\bf p}(x-y) e^{i P(x_1 - y_1, x_2)} \, f(y) \, dy .
$$
For ${\bf p} \in {\mathcal F}_0$, $y_2\in [-1,1]$ and 
$|x_2 - y_2| \sim 2^{q}$, we have $|x_2| \sim 2^{q}$ 
since $q \ge c_{\gamma} \gg 1$. Hence $|(x_2 - y_2)^k - x_2^k| \lesssim 2^{q(k-1)}$,
implying $|P(x_1 - y_1, x_2 - y_2) - P(x_1 - y_1, x_2)| \lesssim |c_{j_0, k_0}| 2^{p j_0 + q (k_0 -1)}$
whenever $\psi_{\bf p}(x - y) \not= 0$ and ${\bf p} = (p,q) \in {\mathcal F}_0$.
Therefore we have
\begin{equation}\label{D-diff}
\|D_{\bf p} a \|_{L^1} \ \lesssim \ |c_{j_0, k_0}| \, 2^{p j_0 + q (k_0 - 1)}
\end{equation}
for any ${\bf p} = (p,q) \in {\mathcal F}_0$ and all atoms $a$ supported in the unit square.

To complement the estimate \eqref{D-diff}, we will observe that the corresponding operator
$$
{\tilde S}_{\bf p} f(x) \ := \ \int_{{\mathbb R}^2} \psi_{\bf p}(x - y) e^{i P(x_1 - y_1, x_2)} 
\varphi(y) f(y) \, dy
$$
for $S_{\bf p}$ satisfies the same $L^2$ operator norm bound as ${\tilde T}_{\bf p}$; namely
\begin{proposition}\label{S*S} Let ${\bf p} = (p, q) \in {\mathbb Z}^2$ be ordered, 
$p\le q$. Then for some $0 < \delta < 1$,
\begin{equation}\label{L2-bound}
\|{\tilde S}_{\bf p}\|_{L^2} \ \lesssim \ 
\begin{cases}
2^{- q/2} \ \ \ |c_{\alpha_{*}} \, 2^{p j_{*} + q k_{*}} |^{-\delta} & \mbox{\rm if} \ p\le 0 \\
2^{-(p+q)/2} \ |c_{\alpha_{*}} \, 2^{p(j_{*} -1) + q k_{*}} |^{-\delta} & \mbox{\rm if} \ p\ge 0 
\end{cases}
\end{equation}
where $P(x) = \sum c_{\alpha} x^{\alpha}$ and  $\alpha_{*} = (j_{*}, k_{*})$ is any element 
in $\Delta_P$ with $j_{*}\ge 1$.  
\end{proposition}

We now deompose ${\mathcal F}_0$ further into a disjoint union 
${\mathcal F}_0 =  {\mathcal F}_{0,+} \cup {\mathcal F}_{0,-}$ where for 
${\bf p} = (p,q) \in {\mathcal F}_{0,+}$, we have $p\ge 0$ and for
${\bf p} = (p,q) \in {\mathcal F}_{0,-}$, we have $p < 0$.
Hence for ${\bf p}\in {\mathcal F}_{0,+}$, $D_{\bf p} a(x)$ is supported in
$$
\bigl\{x = (x_1, x_2): \, |x_1| \sim 2^p \ {\rm and} \ |x_2| \sim 2^{q} \bigr\} 
$$
when $a$ is an atom supported in the unit square. Hence by the Cauchy-Schwarz inequality,
$\|D_{\bf p} a \|_{L^1} \lesssim 2^{(p+q)/2} \|D_{\bf p} a \|_{L^2}$ for ${\bf p} \in {\mathcal F}_{0,+}$.
Also when ${\bf p}\in {\mathcal F}_{0,-}$, $D_{\bf p} a(x)$ is supported in
$$
\bigl\{x = (x_1, x_2): \, |x_1| \lesssim 1 \ {\rm and} \ |x_2| \sim 2^{q} \bigr\} 
$$
and so
$\|D_{\bf p} a \|_{L^1} \lesssim 2^{q/2} \|D_{\bf p} a \|_{L^2}$ for ${\bf p} \in {\mathcal F}_{0,-}$.
Hence applying Propositions \ref{T*T} and \ref{S*S} to the operators $T_{\bf p}$ and
$S_{\bf p}$ separately (recall that we are assuming $j_0 \ge 1$ for the moment) shows us that
\begin{equation}\label{D-decay}
\|D_{\bf p} a\|_{L^1} \ \lesssim \ 
\begin{cases}
|c_{j_0, k_0} 2^{p j_0 + q k_0} |^{-\delta} & \mbox{\rm if} \ p\le 0 \\
|c_{j_0, k_0} 2^{p (j_0 - 1) + q k_0} |^{-\delta} & \mbox{\rm if} \ p\ge 0 
\end{cases}
\end{equation}
which, together with \eqref{D-diff} allows us to successfully sum $\|D_{\bf p} a\|_{L^1}$
over ${\bf p} \in {\mathcal F}_0 = {\mathcal F}_{0,+} \cup {\mathcal F}_{0,-}$. 

To see this, let us treat the cases ${\bf p} \in {\mathcal F}_{0,+}$ and
${\bf p} \in {\mathcal F}_{0,-}$ separately. When ${\bf p} = (p,q) \in {\mathcal F}_{0,+}$, we take
a convex combination of the bounds in \eqref{D-diff} and \eqref{D-decay}; for any $0<\epsilon<1$,
we have 
$$
\|D_{\bf p} a \|_{L^1} \ \lesssim \ 
\frac{|c_{j_0, k_0} 2^{p j_0 + q (k_0 - 1)} |^{\epsilon} }{|c_{j_0, k_0} 
2^{p (j_0 -1) + q k_0}|^{\delta (1-\epsilon)}} .
$$
We choose $\epsilon$ such that
\begin{equation}\label{eps-condition}
\frac{\epsilon}{1-\epsilon} \  > \ \delta \
\frac{j_0 -1}{j_0} \ \ {\rm or} \ \ \epsilon j_0 \ > \delta (1 - \epsilon) (j_0 -1)
\end{equation}
This allows us, for fixed $q$, to sum in $p\le q$ to conclude
$$
\sum_{p \in {\mathcal F}_{0,+}^{q}} \|D_{\bf p} a \|_{L^1} \ \lesssim \
\bigl[ |c_{j_0, k_0}| 2^{(j_0 + k_0  - 1) q} \bigr]^{\epsilon - \delta (1-\epsilon)}
$$
where ${\mathcal F}_{0,+}^{q} = \{ p \in {\mathbb Z} : (p, q) \in {\mathcal F}_{0,+} \}$.

Finally, for $q \ge c_{\gamma} \gg 1$, we split this sum further; when
$|c_{j_0, k_0}| 2^{(j_0 + k_0 -1) q} \le 1$, we choose $\epsilon$
so that $\epsilon/(1-\epsilon) > \delta$ (which implies the condition \eqref{eps-condition})
and this allows us to sum over $q \ge 0$ to obtain an $O(1)$ bound. When 
$|c_{j_0, k_0}| 2^{(j_0 + k_0 -1) q} \ge 1$, we further restrict $\epsilon$
so that $\epsilon/(1-\epsilon) < \delta$. We note that it is possible to choose $\epsilon$
so that
$$
\delta \
\frac{j_0 -1}{j_0} \ < \
\frac{\epsilon}{1-\epsilon} \  < \ \delta
$$
and with this choice, we can successfully sum over these $q \ge 0$ and hence
\begin{equation}\label{D-sum-+}
\sum_{{\bf p}\in {\mathcal F}_{0,+}} \|D_{\bf p} a \|_{L^1} \ \lesssim 1.
\end{equation}

We now turn to bounding the sum  $\sum_{{\bf p} \in {\mathcal F}_{0,-}} \|D_{\bf p} a\|_{L^1}$.
Again we take
a convex combination of the bounds in \eqref{D-diff} and \eqref{D-decay}; for any $0<\epsilon<1$,
we have 
$$
\|D_{\bf p} a \|_{L^1} \ \lesssim \ 
\frac{|c_{j_0, k_0} 2^{p j_0 + q (k_0 - 1)} |^{\epsilon} }{|c_{j_0, k_0} 
2^{p j_0 + q k_0}|^{\delta (1-\epsilon)}} \ = \
\bigl[|c_{j_0, k_0}| 2^{p j_0 + q k_0} \bigr]^{\epsilon - \delta (1-\epsilon)} \, 
2^{-\epsilon q} 
$$
for any ${\bf p} = (p,q) \in {\mathcal F}_{0,-}$. Again we will fix $q$ and sum
over $p\in {\mathcal F}_{0,-}^q$ first. For those $p$ such that
$|c_{j_0, k_0}| 2^{p j_0 + q k_0} \le 1$, we choose $\epsilon$ such that
$\epsilon - \delta (1 - \epsilon) >0$ and for those $p$ such that
$|c_{j_0, k_0}| 2^{p j_0 + q k_0} \ge 1$, we choose $\epsilon$ such that
$\epsilon - \delta (1 - \epsilon) < 0$. In either case we see that
$$
\sum_{p \in {\mathcal F}_{0,-}^{q}} \|D_{\bf p} a \|_{L^1} \ \lesssim \
2^{-\epsilon q}
$$
and so
\begin{equation}\label{D-sum--}
\sum_{{\bf p}\in {\mathcal F}_{0,-}} \|D_{\bf p} a \|_{L^1} \ \lesssim 1.
\end{equation}
The bounds \eqref{D-sum-+} and \eqref{D-sum--} reduce matters (modulo Proposition
\ref{S*S}) to examining
$\| \sum_{{\bf p}\in {\mathcal F}_0} S_{\bf p} a \|_{L^1}$ in the case $j_0 \ge 1$. 

We first consider those ${\bf p} = (p,q) \in {\mathcal F}_{0,+}$ and note that
$$
S_{\bf p} a(x) \ = \ \int_{{\mathbb R}^2} \bigl[ \psi_q (x_2 - y_2) - \psi_q (x_2) \bigr]
\psi_p (x_1 - y_1) e^{i P(x_1 - y_1, x_2)} a(y) \,  dy
$$
by the cancellation property of the atom $a$. Since $|x_2| \sim |x_2 - y_2| \sim 2^q$ when
$\psi_q (x_2 - y_2) \not= 0$ and $y_2 \in [-1,1]$, we have
$$
|S_{\bf p} a(x)| \ \lesssim \ 2^{-2q}\chi_{|x_2| \sim 2^q}(x_2) \int_{{\mathbb R}^2}  
| \psi_p (x_1 - y_1) a(y) | \, dy
$$
and so $\|S_{\bf p} a\|_{L^1} \lesssim 2^{-q}$ implying that
$$
\sum_{{\bf p} = (p,q) \in {\mathcal F}_{0,+}} \|S_{\bf p} a \|_{L^1} \ \lesssim \ 1.
$$
For ${\bf p} = (p,q) \in {\mathcal F}_{0,-}$, we again use the cancellation property
of the atom $a$ to write
$$
\sum_{{\bf p} \in {\mathcal F}_{0,-}} S_{\bf p} a (x) \ = \ \sum_{q\ge 0}
 \int_{{\mathbb R}} \bigl[ \psi_q (x_2 - y_2) - \psi_q (x_2) \bigr] S^{x_2}_q a_{y_2} (x_1) d y_2
 $$
 where $a_{y_2}(u) = a(u, y_2)$ and
 $$
 S^{x_2}_q g (x_1) \ := \ \int_{\mathbb R} 
 \bigl[ \sum_{p: (p,q) \in {\mathcal F}_{0,-}} \psi_p (x_1 - y_1) \bigr]
e^{i P(x_1 - y_1, x_2)} g(y_1) \, dy_1.
$$
The operator $S^{x_2}_q$ is a multiplier operator on ${\mathbb R}$ with multiplier
$$
m^{x_2}_q(\xi) \ = \ \sum_{p: (p,q) \in {\mathcal F}_{0,-}} 
\int_{\mathbb R} \psi_p(s) \ e^{i (P(s, x_2) + \xi s)} \, ds
$$
which by Lemma \ref{UL-2} is a bounded function of $\xi$, uniformly in the parameters
$x_2, q$ and the set ${\mathcal G}_q = \{ p: (p,q) \in {\mathcal F}_{0,-}\}$. Hence $S^{x_2}_q$
is uniformly bounded on $L^2$.

For fixed $x_2$ and $q$,  $S^{x_2}_q a_{x_2} (\cdot)$ is supported in $[-3,3]$ and so
by the Cauchy-Schwarz inequality,
$$
\bigl\|\sum_{{\bf p} \in {\mathcal F}_{0,-}} S_{\bf p} a \bigr\|_{L^1} \ \lesssim \
\sum_{q\ge 0} 2^{-2q} \int_{|x_2| \sim 2^q} \Bigl[\int_{\mathbb R} \| S^{x_2}_q a_{y_2} \|_{L^2} 
dy_2\Bigr] \, dx_2 
$$
$$
\lesssim \ \int_{\mathbb R} \sqrt{\int_{\mathbb R} |a(y_1, y_2)|^2 \, dy_1} \ \ dy_2 \ \lesssim \ 1,
$$
the last inequality following by a final application of the Cauchy-Schwarz inequality. The completes
the proof of Theorem \ref{2-d} in the case $j_0\ge 1$, once Proposition \ref{S*S} is proved.

{\bf Proof of Proposition \ref{S*S}} This proceeds exactly along the lines of Proposition \ref{T*T}
by considering
the kernel $M(x,u)$ of ${\tilde S}_{\bf p}^{*} {\tilde S}_{\bf p}$ which is given by
$$
M(x,u) \ = \ \varphi(x)\varphi(u) \int_{{\mathbb R}^2} e^{i [P(y_1 - x_1, y_2) - P(y_1-u_1, y_2)]} 
\psi_{\bf p}(y-x) \psi_{\bf p}(y-u) \, dy.
$$
We have the same support conditions for $M$ as we did for $L$ and again 
we make the change of variables $y \to 2^{\bf p}\circ y$ to conclude
$$
M(x,u) \ = \ \varphi(x)\varphi(u) 2^{-(p+q)} \int_{{\mathbb R}^2} e^{i \Phi(y)} \, 
\Theta(y) \, dy.
$$
where this time 
$$
\Phi(y) = \Phi_{{\bf p}, x, u}(y) = \sum c_{j,k} 2^{p j + q k} 
\bigl[ (y_1 - 2^{-p} x_1)^{j} - (y_1 - 2^{-p}u_1)^{j} \bigr] y_2^k
$$
and $\Theta(y)$ is unchanged, 
a smooth function, supported in $[-5,5]^2$ with uniformly bounded $C^k$ norms. Using
an appropriately modified defintion of $g(t)$ we see that
$$
\Phi(y) \ = \ \int_0^1 \Bigl[ \sum_{(j,k) \in \Delta_P} c_{j,k} 2^{p j + q k} 
j 2^{-p} (x_1 - u_1) X^{j -1} y_2^k 
 \Bigr] \, dt
$$
$$
= \ \int_0^1 \Bigl[\sum_{(j,k) \in {\tilde \Delta}_P} \bigl[
c_{j+1, k} 2^{j p + k q}
(j+1) (x_1 - u_1) \bigr] X^j Y^k \Bigr] \, dt
$$
where now ${\tilde \Delta}_P =  \Delta_P - (1,0)$ and where 
$X$ is same as before but now $Y = y_2$. Again we see
that $X, Y \in [-10,10]^2$ for all $t\in [0,1]$. 

The analysis now proceeds exactly as before. We only note that (for an appropriately modified $Q$)
$$
\|Q\|_1 \ \ge \ |d_{j_0 - 1, k_0}| \ = \
\bigl| c_{j_0, k_0} 2^{p (j_0 - 1) + q k_0} j_0 
(x_1 - u_1) \bigr|.
$$
The rest of the proof of Proposition \ref{S*S} follows line by line the proof of Proposition \ref{T*T}.

This completes the proof of \eqref{reduction} with ${\mathcal F} = {\mathcal F}_0$ in the
case $j_0 \ge 1$. 

{\bf The case $j_0 = 0$.} When $j_0 = 0$, we modify the above argument as follows.

First we note that $(0, k_0) \in \Delta$ shows that
 $\Delta_0 = \{ k\ge 2 : (0, k) \in \Delta \}$ is nonempty and so, since $|\Delta_0| |\Delta_1| =0$, 
we see that $\Delta_1 = \emptyset$. We decompose ${\mathcal F}_0$ into
$O_{|\Delta|}(1)$ disjoint sets $\{{\mathcal F}_{0,\sigma}\}$ such that for each $\sigma$,
there is an $\alpha_{\sigma} = (j_{\sigma}, k_{\sigma}) \in \Delta \setminus \Delta_0$ with
$j_{\sigma}\ge 2$ (since $\Delta_1 = \emptyset$) and
$$
|c_{\alpha_{\sigma}}| 2^{{\bf p}\cdot \alpha_{\sigma}} \ \ge \ 
|c_{\alpha}| 2^{{\bf p}\cdot \alpha}
$$
for all $\alpha \in \Delta \setminus \Delta_0$ whenever ${\bf p} \in {\mathcal F}_{0,\sigma}$.
Note that $|c_{\alpha_0}| 2^{{\bf p}\cdot \alpha_0} \ge |c_{\alpha_{\sigma}}| 2^{{\bf p}\cdot
\alpha_{\sigma}}$ for all ${\bf p} \in {\mathcal F}_0$.

We fix one of these subsets ${\mathcal F}_{0,\sigma_1}$ and establish \eqref{reduction} with
${\mathcal F} = {\mathcal F}_{0,\sigma_1}$. To simplify notation we write
$\alpha_1 = (j_1, k_1)$ instead of $\alpha_{\sigma_1} = (j_{\sigma_1}, k_{\sigma_1})$.

We write $P(x_1, x_2) = Q(x_1, x_2) + T(x_2)$ where ($x = (x_1, x_2)$)
$$
Q(x) \ = \ \sum_{\alpha \in \Delta\setminus \Delta_0} c_{\alpha} \, x^{\alpha} \ \ \
{\rm and} \ \ \ T(x_2) \ = \ \sum_{(0,k)\in \Delta_0} c_{0,k} \, x_2^k.
$$
We modify the definition of the comparison operator $S_{\bf p}$ as
$$
S_{\bf p} f(x) \ := \  \int_{{\mathbb R}^2} \psi_{\bf p}(x - y) e^{i [ Q(x_1 - y_1, x_2) + T(x_2 - y_2)]}
f(y) \, dy
$$
and consider $D_{\bf p} = T_{\bf p} - S_{\bf p}$ as before. The two estimates \eqref{D-diff}
and \eqref{D-decay} now become
\begin{equation}\label{D-diff-mod}
\|D_{\bf p} a \|_{L^1} \ \lesssim \ |c_{j_1, k_1}| 2^{p j_1 + q (k_1 -1)}
\end{equation}
and
\begin{equation}\label{D-decay-mod}
\|D_{\bf p} a\|_{L^1} \ \lesssim \ 
\begin{cases}
|c_{j_1, k_1} 2^{p j_1 + q k_1} |^{-\delta} & \mbox{\rm if} \ p\le 0 \\
|c_{j_1, k_1} 2^{p (j_1 - 1) + q k_1} |^{-\delta} & \mbox{\rm if} \ p\ge 0 
\end{cases} .
\end{equation}
The difference bound \eqref{D-diff-mod} is straightforward and the decay bound \eqref{D-decay-mod}
follows along the same lines establishing \eqref{D-decay}. The key here is that $j_1 \ge 1$ (in fact we
know $j_1 \ge 2$ and this will be needed later) and so \eqref{D-diff-mod} and \eqref{D-decay-mod}
together allow us to see that the sum 
$\sum_{{\bf p} \in {\mathcal F}_{0,\sigma_1}} \|D_{\bf p} a\|_{L^1}$ is uniformly bounded, 
reducing matters to showing 
\begin{equation}\label{reduction-S}
\int_{{\mathbb R}^2} \bigl| S_{{\mathcal F}_{0,\sigma_1}} a (x) \bigr| \, dx \ \lesssim \ 1
\end{equation}
where
$S_{{\mathcal F}_{0,\sigma_1}} = \sum_{{\bf p} \in {\mathcal F}_{0,\sigma_1}} S_{\bf p}$.
The arguments for the case $j_0 \ge 1$ do not apply to $S_{{\mathcal F}_{0,\sigma_1}}$
and we return to the proof of \eqref{reduction-S} after an interlude.

\section{Proof of Theorem \ref{2-d} -- the proof of the necessity}

Since the proof of the necessity uses some arguments from the previous section, we
now pause in the proof of the sufficiency part and give a proof of the necessity;
that is, to show that condition \eqref{nec-suff} in Theorem \ref{2-d} is a necessary condition
for the uniform bound \eqref{main-bound} to hold.

The necesssity of $(1,0)$ and $(0,1) \notin \Delta$ is well known so we will assume
this condition holds but suppose $|\Delta_0| |\Delta_1| + |\Delta^0| |\Delta^1| \ge 1$. Under these
assumptions we will show that the uniform bound in \eqref{main-bound} does not hold. Without
loss of generality suppose $|\Delta_0| |\Delta_1| \ge 1$ so that there exists $k_0\ge 2$ and
$k_1\ge 1$ such that $(0,k_0), (1, k_1)\in \Delta$. We consider the subfamily of polynomials
$P_{c,d}(s,t) = c s t^{k_1} + d t^{k_0}  \in {\mathcal V}_{\Delta}$ as $c, d$ vary over ${\mathbb R}$.
If \eqref{main-bound} holds, then
\begin{equation}\label{special-main-bound}
\int\!\!\!\int_{10 \le |x_1|\le |x_2| \ll \epsilon^{-1}} \bigl| S_{P_{c,d}, {\mathcal K}} a(x_1, x_2) \bigr|
\, dx_1 dx_2 \ \lesssim \ 1
\end{equation}
holds uniformly for all $0<\epsilon \ll 1, c, d \in {\mathbb R}$ and atoms $a$ supported in the
unit square. 
Our aim is to show that \eqref{special-main-bound} does not hold. 

In fact, for our
atom $a$ we simply take $a(y) = b(y_1) b(y_2)$ where $y = (y_1, y_2)$ and
$b(u) = 1$ when $0 \le u \le 1/2$ and $b(u) = -1$ when $-1/2 \le u < 0$. We will choose
$c = c(\epsilon), d = d(\epsilon) \in {\mathbb R}$ and show that the integral in 
\eqref{special-main-bound},
$$
I(\epsilon) \ := \ 
\int\!\!\!\int_{10 \le |x_1|\le |x_2| \ll \epsilon^{-1}} \Bigl| 
\int\!\!\!\int_{{\mathbb R}^2}  a(x_1 -s, x_2 -t)  e^{i [c s t^{k_1} + d t^{k_0}]} \frac{ds dt}{s t} \Bigr|
\, dx_1 dx_2,
$$
satisfies $I(\epsilon) \gtrsim \log(\epsilon^{-1})$ which will show that \eqref{special-main-bound}
fails.

From the arguments of the previous section, we see that
$$
I(\epsilon) = 
\int\!\!\!\int_{E} \Bigl| \int\!\!\!\int_{{\mathbb R}^2} e^{i[c(x_1 - s) x_2^{k_1} + d (x_2 - t)^{k_0}]}
\frac{1}{(x_1 - s)(x_2 - t)} a(s, t) \, ds dt \Bigr|
\, dx_1 dx_2 \ + \ O(1)
$$
holds where 
$$
E \ :=  \ \bigl\{(x_1, x_2): 10 \le |x_1|\le |x_2| \ll \epsilon^{-1}, |c x_1 x_2^{k_1}| \le |d x_2^{k_0}|
\bigr\}.
$$
This is precisely the reduction to \eqref{reduction-S} when $P(s,t) = c s t^{k_1} + d t^{k_0}$.

We note that
$$
\int\!\!\!\int_{10 \le |x_1|\le |x_2|} \Bigl| 
\int\!\!\!\int_{{\mathbb R}^2} e^{i [c (x_1 - s) x_2^{k_1} + d (x_2 - t)^{k_0}]} 
\frac{1}{x_1 - s} \Bigl[\frac{1}{x_2 - t} - \frac{1}{x_2}\Bigr] a(s,t) ds dt \Bigr|
\, dx_1 dx_2
$$
$$
\lesssim \ \int_{10\le |x_1|} \frac{1}{|x_1|} \Bigl[ \int_{|x_1|\le |x_2|} \frac{1}{x_2^2} \ d x_2 \Bigr] 
\, dx_1 \ \lesssim
\ \int_{10 \le |x_1|} \frac{1}{x_1^2} \ d x_1 \ \lesssim \
\ 1
$$
and so
$$
I(\epsilon) = 
\int\!\!\!\int_{E} \frac{1}{|x_2|} \Bigl| \int\!\!\!\int_{{\mathbb R}^2} 
e^{i[c(x_1 - s) x_2^{k_1} + d (x_2 - t)^{k_0}]}
\frac{1}{(x_1 - s)} a(s, t) \, ds dt \Bigr|
\, dx_1 dx_2 \ + \ O(1).
$$
Next we show that the integral
$$
\int\!\!\!\int_{E} \frac{1}{|x_2|} \Bigl| \int\!\!\!\int_{{\mathbb R}^2} 
e^{i[c(x_1 - s) x_2^{k_1} + d (x_2 - t)^{k_0}]}
\Bigl[\frac{1}{(x_1 - s)} - \frac{1}{x_1}\Bigr] a(s, t) \, ds dt \Bigr|
\, dx_1 dx_2
$$
is $O(1)$. We note that this integral is at most
$$
\int\!\!\!\int_{10\le |x_1|\le |x_2|} \frac{1}{|x_2| x_1^2}  \ \Bigl[
\int_{\mathbb R} \Bigl| \int_{\mathbb R} e^{i d(x_2 - t)^{k_0}} a(s,t) \, dt \Bigr| \, ds \Bigr] \, dx_1 dx_2
\ =: \ I ;
$$
we split $I = II + III$ into two integrals where the integration in $II$ is further restricted 
to where $|d x_2^{k_0 -1}| \le 1$ and the integration in $III$ is over the complement,
where $|d x_2^{k_0 -1}| \ge 1$.

Using the cancellation of the atom $a$, we see that
$$
II \ = \ 
\int\!\!\!\int_{F} \frac{1}{|x_2| x_1^2}  \ \Bigl[
\int_{\mathbb R} \Bigl| \int_{\mathbb R} \bigl[ e^{i d(x_2 - t)^{k_0}} - e^{id x_2^{k_0}}\bigr] 
a(s,t) \, dt \Bigr| \, ds \Bigr] \, dx_1 dx_2
$$
where $F = \{(x_1, x_2) : 10 \le |x_1|\le |x_2|,  |d x_2^{k_0}| \le 1 \}$. Hence
$$
II \lesssim |d| \, \int_{10\le |x_1|} \frac{1}{x_1^2} \Bigl[ \int_{|d x_2^{k_0}|\le 1}
|x_2|^{k_0 -2} \, d x_2 \Bigr] \, dx_1 \ \lesssim \ 1
$$
where we used (crucially) the fact that $k_0 \ge 2$. 

To treat $III$, we fix $s$ and write $B(x_2) = B(x_2, s) = \int_{\mathbb R} e^{i d(x_2 - t)^{k_0}} a(s,t) dt$
for the inner integral in $III$ and 
use the Cauchy-Schwarz inequality
to see that
$$
\int_{1\le |x_2|, |d x_2^{k_0-1}|} \frac{1}{|x_2|} |B(x_2)| dx_2 \ \le
$$
$$
\sqrt{\int_{1\le |x_2|} \frac{1}{|x_2|^{1 + 1/k_0}} } \ \ \ \sqrt{\int_{1\le |d x_2^{k_0 -1}|}
 \frac{1}{|x_2|^{(k_0 -1)/k_0}} |B(x_2)|^2 dx_2} \ \ \lesssim 
$$
$$
|d|^{1/k_0} \sqrt{ \int_{\mathbb R} \Bigl| \int_{\mathbb R} e^{i d (x_2 - t)^{k_0}} a(s,t) \, dt \Bigr|^2 dx_2}
\ = \ |d|^{1/k_0} \sqrt{ \int_{\mathbb  R} |{\tilde a}(s,\eta) m(\eta)|^2 d\eta }
$$
where ${\tilde a}$ denotes the partial fourier transform in the second variable. Here
$$
m(\eta) \ = \ \int_{\mathbb R} e^{i [d t^{k_0} + \eta t]}\  dt
$$
is the oscillatory integral multiplier which arises when computing the fourier transform of
$B(x_2)$ and is a well defined integral (defined as a limit as $R\to\infty$ of truncated
integrals $|t|\le R$ which converges since $k_0 \ge 2$). Furthermore by van der Corput's
lemma, see \cite{S} page 332, we have $|m(\eta)| \le C_{k_0} |d|^{-1/k_0}$ and so
$$
\int_{1\le |x_2|, |d x_2^{k_0-1}|} \frac{1}{|x_2|} |B(x_2)| dx_2 \ \lesssim \
\sqrt{\int_{\mathbb R} |a(s, t)|^2 \, dt}
$$
implying that
$$
|III| \ \lesssim \
\int_{1\le |x_1|} \frac{1}{x_1^2} \int_{\mathbb R} 
\sqrt{\int_{\mathbb R} |a(s, t)|^2 \, dt} \ ds \ dx_1 \ \lesssim \ 1
$$
with a final application of the Cauchy-Schwarz inequality.

Therefore
$$
I(\epsilon) \ = \ \int\!\!\!\int_E \frac{1}{|x_1||x_2|} \Bigl| \int_{{\mathbb R}^2} 
e^{i [-c s x_2^{k_1} + d (x_2 - t)^{k_0}]} a(s,t) \, ds dt \Bigr| \, dx_1 dx_2 \ + \ O(1).
$$
From our definition of $a(s,t) = b(s) b(t)$ we have
$$
\int_{{\mathbb R}^2} 
e^{i [-c s x_2^{k_1} + d (x_2 - t)^{k_0}]} a(s,t) \, ds dt \ = \ {\widehat b}(c x_2^{k_1}) \,
\int_{|t|\le 1/2} b(t) e^{i d (x_2 - t)^{k_0}} dt
$$
and we note that 
\begin{equation}\label{FT-b}
|{\widehat b}(c x_2^{k_1})| \ = \ \Bigl| \frac{\cos(c x_2^{k_1}) - 1}{c x_2^{k_1}} \Bigr| \ \gtrsim \
|c x_2^{k_1}|
\end{equation}
whenever $|c x_2^{k_1}| \ll 1$. With a little work, we can also show that
\begin{equation}\label{b}
\Bigl| \int_{|t|\le 1/2} b(t) e^{i d (x_2 - t)^{k_0}} \, dt \Bigr| \ \gtrsim \
|d x_2^{k_0 - 1}|
\end{equation}
whenever $|d x_2^{k_0 -1}| \ll 1 \ll |x_2|$. We will show this later.

Hence
$$
I(\epsilon) \ \ \gtrsim \ \  \int\!\!\!\int_G \frac{1}{|x_1| |x_2|} \ |d x_2^{k_0 -1}| \, |c x_2^{k_1}| \, dx_1 dx_2
$$
where $G = \{(x_1, x_2) \in E: |d x_2^{k_0 -1}|, |c x_2^{k_1}| \ll 1 \}$. We now choose
$d = \epsilon^{k_0 - 1}$ and $c = \epsilon^{k_1}$ so that
$|x_2| \ll \epsilon^{-1}$ implies $|c x_2^{k_1}|, |d x_2^{k_0 - 1}| \ll 1$. Thus with this
choice of $c$ and $d$, we have $E = G$. We divide the concluding analysis of the integral above
into two cases.

{\bf Case 1}: $k_0 -1 \le k_1$. Here $|x_1| \le |x_2| \le \epsilon^{-1}$ automatically implies 
$|c x_1 x_2^{k_1}| \le |d x_2^{k_0}|$ and so $E = G = \{(x_1,x_2) : 10 \le |x_1| \le |x_2| \ll \epsilon^{-1}\}$.
In fact,
$$
|x_2|\le \epsilon^{-1} \ \Rightarrow \ |\epsilon x_2|^{k_1 - k_0 +1} \le 1 \ \Rightarrow \
|\epsilon x_2|^{k_1} \le |\epsilon x_2|^{k_0 -1} \ \Rightarrow \
|\epsilon x_2|^{k_1} |x_2| \le |d x_2^{k_0}|
$$
and so
$$
|c x_1 x_2^{k_1}| \ = \ |\epsilon x_2|^{k_1} |x_1| \ \le \ |\epsilon x_2|^{k_1} |x_2| \ \le \
|d x_2^{k_0}|.
$$
Therefore
$$
I(\epsilon) \ \gtrsim \ \int\!\!\!\int_{E} 
\frac{1}{|x_1| |x_2|} |\epsilon x_2|^{k_0 + k_1 - 1} \, dx_1 dx_2 \ \gtrsim \
\int_{10\le |x_2| \ll \epsilon^{-1}} \log(|x_2|) \ |\epsilon x_2|^{k_0 + k_1 -2} \epsilon dx_2
$$
$$
\sim \ \int_{10 \epsilon \le y \ll 1} \log(y/\epsilon) \ y^{k_0 + k_1 - 2} \, dy \ \gtrsim \
\log(1/\epsilon) 
$$
since $k_0 + k_1 - 2 \ge 1$. 

{\bf Case 2}: $k_1 < k_0 - 1$. Here $|c x_1 x_2^{k_1}| \le |d x_2^{k_0}|$ and $|x_2| \le \epsilon^{-1}$
imply that $|x_1| \le |x_2|$. In fact,
$$
|c x_1 x_2^{k_1}| \le |d x_2^{k_0}| \ \Rightarrow \ |x_1| |\epsilon x_2|^{k_1} \le |x_2| 
|\epsilon x_2|^{k_0 - 1}
$$
and since $|x_2| \le \epsilon^{-1}$,  we have $|x_1| \le |\epsilon x_2|^{k_0 - k_1 -1} |x_2| \le |x_2|$.
Hence 
$$
E \ = \ \bigl\{(x_1, x_2): 10 \le |x_2| \ll \epsilon^{-1}, 10\le |x_1|, \ {\rm and} \
 |x_1| |\epsilon x_2|^{k_1} \le
|x_2| |\epsilon x_2|^{k_0 -1} \bigr\}
$$
and so
$$
I(\epsilon) \ \  \gtrsim \ \ \int\limits_{\begin{array}{c}\scriptstyle
10\le |x_1|, 10\le |x_2| \ll \epsilon^{-1}\\
         \vspace{-5pt}\scriptstyle |\epsilon x_2|^{k_1} |x_1| \le |\epsilon x_2|^{k_0 -1} |x_2|
         \end{array}}
\frac{1}{|x_1| |x_2|} \ |\epsilon x_2|^{k_0 + k_1 - 1} \, dx_1 dx_2
$$
$$
\gtrsim \ \int\limits_{\begin{array}{c}\scriptstyle
10 \epsilon\le |y| \ll 1\\
         \vspace{-5pt}\scriptstyle 10 \le |x_1| \le |y|^{k_0 - k_1 -1} |y/\epsilon| 
         \end{array}}
\frac{1}{|x_1|} \ |y|^{k_0 + k_1 - 2} \, dx_1 dy
$$
$$
\gtrsim \ \int_{\epsilon^{1/(k_0 - k_1)} \le |y| \ll 1} \bigl[ \log(1/\epsilon) + \log(|y|^{k_0 - k_1}) \bigr]
|y|^{k_0 + k_1 -2} \, dy \ \gtrsim \ \log(1/\epsilon)
$$
which shows that $I(\epsilon) \gtrsim \log(1/\epsilon)$ holds in both cases IF
\eqref{b} holds. 

We now establish \eqref{b}. First we note that
$$
\int_{|t|\le 1/2} b(t) e^{i d (x_2 - t)^{k_0}} \, dt \ = \
\int_0^{1/2} \bigl[ e^{i d (x_2 - t)^{k_0}} - e^{i d (x_2 + t)^{k_0}} \bigr] \, dt
$$
$$
= \ 
e^{i d x_2^{k_0}} \int_0^{1/2} \bigl[ e^{i [ k_0 d x_2^{k_0 -1} (-t) + \ldots ]} - 
e^{i [ k_0 d x_2^{k_0 -1} t + \ldots ] } \bigr] \, dt
$$
and so
$$
\Bigl| \int_{|t|\le 1/2} b(t) e^{i d (x_2 - t)^{k_0}} \, dt \Bigr|  \ge 
\Bigl| \int_0^{1/2} \bigl( \sin(k_0 d x_2^{k_0 -1}  t [ 1 + g(t) ]) +
\sin(k_0 d x_2^{k_0 -1}  t [ 1 + h(t) ]) \bigr) \, dt \Bigr|
$$
where $g(t), h(t) = O(1/|x_2|)$. For large $|x_2| \gg 1$ and small $|d x_2^{k_0 -1}| \ll 1$, we see that
the integrand in the above integral is single-signed and both
$$
\sin(k_0 d x_2^{k_0 -1}  t [ 1 + g(t) ]),  \sin(k_0 d x_2^{k_0 -1}  t [ 1 + h(t) ])| \
 = \ k_0 d x_2^{k_0 -1} ( 1 + F(t))
$$
for some $|F(t)| \le 1/2$  and all $0\le t \le 1/2$. Hence
$$
|\sin(k_0 d x_2^{k_0 -1}  t [ 1 + g(t) ]) +
\sin(k_0 d x_2^{k_0 -1}  t [ 1 + h(t) ])| \ \gtrsim \ |d x_2^{k_0 -1}|
$$
for all $0\le t \le 1/2$, showing that indeed \eqref{b} holds.

\section{Proof of Theorem \ref{2-d} -- the conclusion of the sufficiency part}

We return to complete the proof of the sufficiency part ot Theorem \ref{2-d} where
matters were reduced to establshing \eqref{reduction-S}.

We split ${\mathcal F}_{0,\sigma_1}$ into  ${\mathcal F}_{0,\sigma_1}^+ \cup
{\mathcal F}_{0,\sigma_1}^{-}$ where
$$
{\mathcal F}_{0,\sigma_1}^+ \ := \ \bigl\{(p,q) \in {\mathcal F}_{0,\sigma_1} : p\ge 0\bigr\}
\ \ {\rm and} \ \
{\mathcal F}_{0,\sigma_1}^{-} \ := \ \bigl\{(p,q) \in {\mathcal F}_{0,\sigma_1} : p< 0\bigr\}.
$$
We first concentrate on establishing \eqref{reduction-S} for ${\mathcal F}_{0,\sigma_1}^+$. We
further split ${\mathcal F}_{0,\sigma_1}^+$ into 
${\mathcal F}_{0,\sigma_1}^{+,1} \cup
{\mathcal F}_{0,\sigma_1}^{+,2}$ where
$$
{\mathcal F}_{0,\sigma_1}^{+,1} \ := \ \bigl\{(p,q) \in {\mathcal F}_{0,\sigma_1}^+ :
|c_{j_1, k_1}| 2^{ p(j_1 - 1) + q k_1} \le |c_{0,k_0}| 2^{q(k_0 -1)} \bigr\}
$$
and ${\mathcal F}_{0,\sigma_1}^{+,2}$ is defined with the opposite inequality. We recall that
by condition \eqref{nec-suff}, $j_1 \ge 2$ and now this becomes important in our analysis. 

\subsection*{The bound \eqref{reduction-S} for ${\mathcal F}_{0,\sigma_1}^{+,1}$}
For
${\bf p} = (p,q) \in {\mathcal F}_{0,\sigma_1}^{+,1}$, consider 
$D_{\bf p}^{+,1} = S_{\bf p} - R_{\bf p}^{+,1}$
where 
$$
R_{\bf p}^{+,1} a (x) \ := \ \int_{{\mathbb R}^2} \psi_{\bf p} (x -y) e^{i [Q(x_1, x_2) + T(x_2 - y_2)]}
a (y) \, dy.
$$
For $|x_1 - y_1| \sim 2^p, p\ge 0$ and $|y_1|\le 1$, we have $|x_1| \lesssim 2^p$
and so $| (x_1 - y_1)^j - x_1^j| \lesssim 2^{p(j-1)}$ for any $j\ge 0$ implying
$|Q(x_1 -y_1, x_2) - Q(x_1, x_2)| \lesssim |c_{j_1, k_1}| 2^{p (j_1 -1) + q k_1}$
whenever $\psi_{\bf p}(x-y) \not= 0$ and ${\bf p} = (p,q) \in {\mathcal F}_{0,\sigma_1}^+$.
Therefore
\begin{equation}\label{D+1-diff}
\|D_{\bf p}^{+,1} a \|_{L^1} \ \lesssim \ |c_{j_1, k_1}| 2^{p (j_1 -1) + q k_1}
\end{equation}
holds for any ${\bf p} = (p,q) \in {\mathcal F}_{0,\sigma}^{+,1}$. The complementary decay
bound (established separately for $S_{\bf p}$ and $R_{\bf p}^{+,1}$) is
\begin{equation}\label{D+1-decay}
\|D_{\bf p}^{+,1} a \|_{L^1} \ \lesssim \ \bigl[ |c_{0,k_0}| 2^{q(k_0 -1)} \bigr]^{-\delta}
\end{equation}
which holds for some $0<\delta < 1$ and every ${\bf p} = (p,q) \in {\mathcal F}_{0,\sigma}^{+,1}$.
Let us first see how to combine \eqref{D+1-diff} and \eqref{D+1-decay} to successfully sum over
${\bf p} = (p,q) \in {\mathcal F}_{0,\sigma_1}^{+,1}$.  

For any $0<\epsilon < 1$, we have
$$
\|D_{\bf p} a\|_{L^1} \ \lesssim \ 
\bigl[ |c_{j_1, k_1}| 2^{p(j_1 -1) + q k_1} \bigr]^{\epsilon}
\bigl[ |c_{0,k_0}| 2^{q (k_0-1)}\bigr]^{- \delta(1-\epsilon)}
$$
and so for fixed $q$, we can sum over $p \in H_q := \{p: (p,q) \in {\mathcal F}_{0,\sigma_1}^{+,1}\}$
(using in a crucial way that $j_1 \ge 2$!),
$$
\sum_{p\in H_q} \|D_{\bf p}^{+,1} a \|_{L^1} \ \lesssim \ 
\bigl[ |c_{0,k_0}| 2^{q (k_0-1)}\bigr]^{\epsilon - \delta(1-\epsilon)}
$$
and this can be summed successfully over $q$ because (importantly) 
$k_0\ge 2$; when summing over $q$ such that
$|c_{0,k_0}| 2^{q (k_0-1)} \le 1$, we choose $\epsilon$ such that $\epsilon - \delta(1-\epsilon)>0$
and when summing over $q$ such that
$|c_{0,k_0}| 2^{q (k_0-1)} \ge 1$, we choose $\epsilon$ such that $\epsilon - \delta(1-\epsilon) < 0$.

Hence to establish \eqref{reduction-S} for ${\mathcal F}_{0,\sigma_1}^{+,1}$, matters are reduced
to showing
\begin{equation}\label{one-parameter}
\int_{{\mathbb R}^2} \Bigl| \sum_{{\bf p} \in {\mathcal F}_{0,\sigma_1}^{+,1}} 
\int_{{\mathbb R}^2} \psi_{\bf p}(x - y) e^{i T(x_2 - y_2)} a(y) \, dy \Bigr| \, dx \ \lesssim \ 1
\end{equation}
but this is more or less a one parameter operator and the arguments in \cite{Pan} can be
used to establish \eqref{one-parameter}. 

We now turn to the proof of \eqref{D+1-decay}. By Cauchy-Schwarz, we have
$$
\|D_{\bf p}^{+,1} a \|_{L^1} \ \lesssim \ 2^{(p+	q)/2} \|D_{\bf p} a \|_{L^2}
$$
and to bound $\|D_{\bf p} a \|_{L^2}$, we treat $S_{\bf p}$ and $R_{\bf p}^{+,1}$ separately
by estimating the $L^2$ operator norms of
$$
{\tilde S}_{\bf p} f(x) \ := \ \int_{{\mathbb R}^2} \psi_{\bf p}(x-y) e^{i[Q(x_1 - y_1, x_2) + T(x_2 - y_2)]}
\varphi(y) \, f(y) \, dy
$$
and
$$
{\tilde R}_{\bf p}^{+,1} f(x) \ := \ 
\int_{{\mathbb R}^2} \psi_{\bf p}(x-y) e^{i T(x_2 - y_2)}
\varphi(y) \, f(y) \, dy
$$
via examining the kernels of ${\tilde S}_{\bf p}^{*} {\tilde S}_{\bf p}$ and
${{\tilde R}_{\bf p}^{+,1 \,  *}} {\tilde R}_{\bf p}^{+,1}$.

Instead of the unorthodox argument used in Section \ref{prelims} to combine
Lemmas \ref{equiv-norms} and \ref{UL-1}, we will take a more direct route. The
kernel of ${\tilde S}^{*}_{\bf p} {\tilde S}_{\bf p}$ is
$$
N(x,u) \ = \ \varphi(x)\varphi(u) \int_{{\mathbb R}^2} 
e^{i [Q(y_1 - x_1, y_2) - Q(y_1-u_1, y_2) + T(y_2 - x_2) - T(y_2 - u_2)]} 
\psi_{\bf p}(y-x) \psi_{\bf p}(y-u) \, dy
$$
and again
we make the change of variables $y \to 2^{\bf p}\circ y$ to conclude
$$
N(x,u) \ = \ \varphi(x)\varphi(u) 2^{-(p+q)} \int_{{\mathbb R}^2} e^{i \Phi(y)} \, 
\Theta(y) \, dy.
$$
where now
$$
\Phi(y) \ = \ \Phi_{{\bf p}, x, u}(y)  \ = \
\sum_{(j,k)\in \Delta\setminus \Delta_0} c_{j,k} 2^{p j + q k} 
\bigl[ (y_1 - 2^{-p}x_1)^{j} - (y_1 - 2^{-p} u_1)^{j} \bigr] y_2^k
$$
$$
+ \ \sum_{(0,k) \in \Delta_0} c_{0,k} 2^{q k} \bigl[(y_2 - 2^{-q} x_2)^k - (y_2 - 2^{-q} u_2)^k \bigr]
\ = \ \sum d_{j,k} \, y_1^j y_2^k
$$
and $\Theta(y)$ is the same as before, a
smooth function, supported in $[-5,5]^2$ with uniformly bounded $C^k$ norms. 
We apply Lemma \ref{equiv-norms} directly to $\Phi$ to find a derivative $\partial^{\alpha}$
where $\alpha = (j,k)$ with $|\alpha| = j + k \ge 1$ and such that $|\partial^{\alpha} \Phi (y)|
\gtrsim \|\Phi\|_1$, uniformly for $y \in [-5,5]^2$. This is a case where we will be able
to effectively bound $\|\Phi\|_1$ from below and then a standard multidimensional
version of van der Corput's lemma (as in \cite{S}) suffices athough one can also appeal
to Lemma \ref{UL-1}. 

We note that $\|\Phi\|_1 \ge |d_{0, k_0 -1}|$ and
$$
d_{0, k_0 -1} = 
\sum_{k\ge k_0} c_{0,k} 2^{q k} e_{k,k_0} \big[ (- 2^{-q} x_2)^{k-k_0+1} - (-2^{-q} u_2)^{k-k_0 + 1}\bigr]
$$
$$
+ \ \sum_{j\ge 2} c_{j, k_0 -1} 2^{p j + q (k_0 -1)} \big[ (-2^{-p} x_1)^j - (-2^{-p} u_1)^j \bigr]
$$
where $e_{k,k_0}$ are numerical constants depending only on $k$ and $k_0$. Hence
$$
d_{0, k_0 -1} \ = \
c_{0, k_0} 2^{q(k_0 -1)} 
\Bigl[ u_2 - x_2 + \sum_{k\ge k_0 +1} \frac{c_{0,k}}{c_{0,k_0}} \bigl[ (-x_2)^{k-k_0+1} - (-u_2)^{k-k_0 +1}
\bigr]  
$$ 
$$
+ \ 
\sum_{j\ge 2} \frac{c_{j, k_0 -1}}{c_{0, k_0}} \bigl[ (-x_1)^j - (-u_1)^j \bigr] \Bigr]
$$
and so $d_{0, k_0 -1} = c_{0, k_0} 2^{q (k_0 -1)} [ u_2 - x_2 + O(2^{-q})]$. In fact
$d_{0, k_0 -1} = c_{0, k_0} 2^{q (k_0 -1)} f(x_2)$ where
$f = f_{x_1, u}$ satisfies $|f(x_2)| \lesssim 1$ and $|f'(x_2)|\gtrsim 1$
on $[-3,3]$. 

Hence by Lemma \ref{UL-1} we can find a $0 < \delta < 1$ such that
$$
\int_{|x|\le 1} |N(x,u)| dx \ \lesssim \ 2^{-(p+q)} \frac{1}{|c_{0, k_0} 2^{q (k_0 -1)} |^{\delta}}
\int_{|x|\le 1} \frac{1}{|f(x_2)|^{\delta}} \, dx
$$
and from the properties of $f$ ($\|f\|_{\infty} \lesssim 1$ and $|f'(s)|\gtrsim 1$), it is a
standard argument to show that the integral on the right hand side above is uniformly bounded.
In fact we fix $x_1$ and bound the integral in $x_2$;
$$
\int_{|x_2|\le 1} \frac{1}{|f(x_2)|^{\delta}} \, dx_2 \ = \ \sum_{\ell\ge 0} \ \int_{E_{\ell}} 
\frac{1}{|f(x_2)|^{\delta}} \, dx_2
$$
where $E_{\ell} = \{|x_2|\le 1: |f(x_2)| \sim 2^{-\ell} \}$. Since the derivative of $f$
is bounded below, we see that $|E_{\ell}| \lesssim 2^{-\ell}$ and so
$$
 \sum_{\ell\ge 0} \int_{E_{\ell}} 
\frac{1}{|f(x_2)|^{\delta}} \, dx_2 \ \sim \ \sum_{\ell\ge 0} 2^{\delta \ell} |E_{\ell}| \ \lesssim \
\sum_{\ell\ge 0} 2^{-\ell(1-\delta)}
$$
which converges since $\delta < 1$. For a general treatment of integrals using this method, 
see for example  \cite{RS2}. 

As before this leads to the bound 
$\|{\tilde S}_{\bf p}\|_{L^2 \to L^2} \lesssim 
2^{-(p+q)/2} \bigl[|c_{0, k_0}| 2^{q (k_0 -1)} \bigr]^{-\delta/2}$ which shows that
\eqref{D+1-decay} holds for $S_{\bf p}$. The treatment for ${\tilde R}_{\bf p}^{+,1}$
is easier as the phase function $\Phi(y)$ which arises does not have any terms with
$c_{j, k}$ where $j\ge 2$. 

\subsection*{The bound \eqref{reduction-S} for ${\mathcal F}_{0,\sigma_1}^{+,2}$}
For ${\bf p} \in {\mathcal F}_{0,\sigma_1}^{+,2}$ we consider 
$D_{\bf p}^{+,2} = S_{\bf p} - R_{\bf p}^{+,2}$ where
$$
R_{\bf p}^{+,2} a (x) \ := \ \int_{{\mathbf R}^2} \psi_{\bf p}(x -y) e^{i [ Q(x_1 - y_1, x_2) + T(x_2)]}
a(y) \, dy
$$
so that 
\begin{equation}\label{D+2-diff}
\|D_{\bf p}^{+,2} a \|_{L^1} \ \lesssim \ |c_{0, k_0}| \, 2^{q(k_0 -1)}
\end{equation}
holds for every ${\bf p} = (p,q) \in {\mathcal F}_{0,\sigma_1}^{+,2}$. The
complementary decay bound is
\begin{equation}\label{D+2-decay}
\|D_{\bf p}^{+,2} a \|_{L^1} \ \lesssim \ \bigl[|c_{j_1, k_1}| \, 2^{p (j_1 -1) + q k_1}\bigr]^{-\delta}
\end{equation}
which holds for some $0<\delta<1$. Combing the bounds \eqref{D+2-diff} and \eqref{D+2-decay},
using $j_1 \ge 2$ and $k_0\ge 2$, gives the uniform bound $\sum_{{\bf p} \in {\mathcal F}_{0,\sigma_1}^{+,2}}
\|D_{\bf p}^{+,2} a \|_{L^1} \lesssim 1$ as before.

Hence the proof that \eqref{reduction-S} holds for ${\mathcal F}_{0,\sigma_1}^{+,2}$ reduces to
showing
$$
\int_{{\mathbb R}^2} \ \Bigl| \sum_{{\bf p} \in {\mathcal F}_{0,\sigma_1}^{+,2}} 
\int_{{\mathbb R}^2} \psi_{\bf p}(x - y) e^{i Q(x_1 - y_1, x_2)} a(y) \, dy \Bigr| \ dx \ \lesssim \ 1
$$
but this is precisely the same bound as for $\|\sum_{{\bf p}\in {\mathcal F}_0} S_{\bf p} a \|_{L^1}$
which was treated in Section \ref{prelude} in the case $j_0 \ge 1$ (but now $j_1 \ge 1$). 

It remains to establish \eqref{D+2-decay}. But this is entirely analogous to the bound
\eqref{L2-bound} in Proposition \ref{S*S} for the case $p\ge 0$; we omit the details.

We complete the proof of Theorem \ref{2-d} by establishing \eqref{reduction-S} for
${\mathcal F} = {\mathcal F}_{0,\sigma_1}^{-}$. 

We first consider 
$$
U_{(p,q)} a(x) \ := \ \int_{{\mathbb R}^2} \psi_p (x_1 - y_1) [ \psi_q (x_2 - y_2) - \psi_q(x_2) ]
e^{i [Q(x_1 - y_1, x_2) + T(x_2 - y_2)]} a(y) \, dy
$$
and bound
$$
\int_{{\mathbb R}^2} \bigl| \sum_{(p,q) \in{\mathcal F}_{0,\sigma_1}^{-}} U_{(p,q)} a(x) \bigr|\, dx \
\le \ \int_{{\mathbb R}} \sum_{q\ge c_{\gamma}} \int_{\mathbb R} |\psi_q(x_2 - y_2) - \psi_q(x_2)| \
I(y_2, x_2) \, dy_2 \ dx_2
$$
where
$$
I(y_2, x_2) \ := \ \int_{\mathbb R} \ \Bigl| \sum_{p: (p,q) \in {\mathcal F}_{0,\sigma_1}^{-}} \int_{\mathbb R}
\psi_p (x_1 - y_1) e^{i Q(x_1 - y_1, x_2)} a(y) \ dy_1 \Bigr| \ dx_1 .
$$
We write $a_{y_2} (y_1) = a(y_1, y_2)$ and use the Cauchy-Schwarz inequality, together with
Plancherel's theorem, to see that
$$
I(y_2, x_2) \ \le \ \sqrt{\int_{\mathbb R} |{\widehat a_{y_2}}(\xi) m(\xi)|^2 \, d\xi }
$$
where 
$$
m(\xi) \ = \ \sum_{p: (p,q) \in {\mathcal F}_{0,\sigma_1}^{-}} \int_{\mathbb R}
\psi_p (s) e^{i [Q(s, x_2) + \xi s]} \, ds
$$
satisfies $|m(\xi)| \lesssim 1$ by Lemma \ref{UL-2}. Hence by Plancherel,
$$
\int_{{\mathbb R}^2} \bigl| \sum_{(p,q) \in{\mathcal F}_{0,\sigma_1}^{-}} U_{(p,q)} a(x) \bigr|\, dx \
\lesssim \ \sum_{q\ge c_{\gamma}} 2^{-2q} \int_{|x_2|\sim 2^q} \Bigl[ \int_{\mathbb R}
\sqrt{\int_{\mathbb R} |a(y_1, y_2)|^2 dy_1 } \ dy_1 \Bigr] \, dx_2
$$
and a final application of the Cauchy-Schwarz inequality shows that
$$
\int_{{\mathbb R}^2} \bigl| \sum_{(p,q) \in{\mathcal F}_{0,\sigma_1}^{-}} U_{(p,q)} a(x) \bigr|\, dx \
\lesssim \ 1.
$$
We are left with bounding the $L^1$ norm of $\sum_{(p,q) \in {\mathcal F}_{0,\sigma_1}^{-}}
V_{(p,q)} a(x)$ where
$$
V_{(p,q)} a(x) \ := \ \psi_q (x_2) \, \int_{{\mathbb R}^2}
e^{i[Q(x_1 - y_1, x_2) + T(x_2 - y_2)]} \psi_p (x_1 - y_1) a(y) \, dy.
$$
To do this, we split ${\mathcal F}_{0,\sigma_1}^{-}$ into ${\mathcal F}_{0,\sigma_1}^{-,1}
\cup {\mathcal F}_{0,\sigma_1}^{-,2}$ where
$$
{\mathcal F}_{0,\sigma_1}^{-,1} \ = \ \{(p,q) \in {\mathcal F}_{0,\sigma_1}^{-} :
|c_{j_1, k_1}| 2^{p j_1 + q k_1} \le |c_{0, k_0}| 2^{q (k_0 -1)} \}
$$
and ${\mathcal F}_{0,\sigma_1}^{-,2}$ is defined similarly with the opposite inequality
holding. 

\subsection*{The proof of \eqref{reduction-S} for ${\mathcal F}_{0,\sigma_1}^{-,1}$}
For ${\bf p} = (p,q) \in {\mathcal F}_{0,\sigma_1}^{-,1}$, we consider 
$D_{\bf p}^{-,1} = V_{\bf p} - R_{\bf p}^{-,1}$ where 
$$
R_{\bf p}^{-,1} a(x) \ := \ \psi_q(x_2) \int_{{\mathbb R}^2} e^{i T(x_2 - y_2)} \psi_p(x_1 - y_1)
a(y) \, dy
$$
so that 
\begin{equation}\label{D-1-diff}
\|D_{\bf p}^{-,1} a \|_{L^1} \ \lesssim \ |c_{j_1, k_1} |  \ 2^{p j_1 + q k_1} .
\end{equation}
We also have the complementary decay bound
\begin{equation}\label{D-1-decay}
\|D_{\bf p}^{-,1} a \|_{L^1} \ \lesssim \ \bigl[ |c_{0, k_0}| 2^{q (k_0 - 1)} \bigr]^{-\delta}
\end{equation}
which holds for some $0<\delta < 1$ and every ${\bf p} = (p,q) \in {\mathcal F}_{0,\sigma_1}^{-,1}$. 
This follows in the same way as before, using a $T^{*}T$ argument. Also
as before, using \eqref{D-1-diff} and \eqref{D-1-decay}, we can sum $\|D_{\bf p}^{-,1} a\|_{L^1}$ uniformly over 
${\bf p} \in {\mathcal F}_{0,\sigma_1}^{-,1}$ since $k_0 \ge 2$. 
To complete the proof of \eqref{reduction-S} for ${\mathcal F} = {\mathcal F}_{0,\sigma_1}^{-,1}$
we need to bound the $L^1$ norm of 
$\sum_{{\bf p}\in {\mathcal F}_{0,\sigma_1}^{-,1}} R_{\bf p}^{-,1} a(x)$ but once again, this
acts like a one parameter operator and the arguments of \cite{Pan} apply here. 

\subsection*{The proof of \eqref{reduction-S} for ${\mathcal F}_{0, \sigma_1}^{-,2}$}
Finally we show that \eqref{reduction-S} holds for 
${\mathcal F} = {\mathcal F}_{0,\sigma_1}^{-,2}$. From above, matters are reduced to showing
\begin{equation}\label{V}
\bigl\| \sum_{{\bf p} \in {\mathcal F}_{0,\sigma_1}^{-,2}} V_{\bf p} a \ \bigr\|_{L^1} \ \lesssim \ 1.
\end{equation}
Here we do not need to compare $V_{\bf p}$ with another operator; instead we use the cancellation of the atom $a$ to note that
$$
V_{(p,q)} a(x) \ = \ \psi_q(x_2) \int_{{\mathbb R}^2} e^{i Q(x_1 - y_1, x_2)} \bigl[
e^{i T(x_2 - y_2)} - e^{i T(x_2)} \bigr] \psi_p (x_1 - y_1) a(y) \, dy
$$
and so
\begin{equation}\label{V-diff}
\|V_{(p,q)} a \|_{L^1} \ \lesssim \ | c_{0, k_0}| 2^{q (k_0 -1) } 
\end{equation}
holds for all ${\bf p} = (p,q) \in {\mathcal F}_{0,\sigma_1}^{-,2}$. The complementay
decay bound (which follows by employing the $T^{*}T$ argument as before) is
\begin{equation}\label{V-decay}
\|V_{(p,q)} a \|_{L^1} \ \lesssim \ \bigl[ | c_{j_1, k_1}| 2^{p j_1 + q k_1 } \bigr]^{-\delta}
\end{equation}
and this holds for some $0<\delta < 1$ and all ${\bf p} = (p,q) \in {\mathcal F}_{0,\sigma_1}^{-,2}$. 
The bounds \eqref{V-diff} and \eqref{V-decay} imply, using $k_0 \ge 2$, that
$$
\sum_{{\bf p}\in {\mathcal F}_{0,\sigma_1}^{-,2}} \|V_{(p,q)} a \|_{L^1} \ \lesssim \ 1
$$
which implies \eqref{V}
and this completes the proof of Theorem \ref{2-d}.

\end{document}